\documentclass{amsart}
\usepackage{color}

\newtheorem{theorem}{Theorem}[section]
\newtheorem{lemma}[theorem]{Lemma}
\newtheorem{proposition}[theorem]{Proposition}
\newtheorem{conjecture}[theorem]{Conjecture}
\newtheorem{corollary}[theorem]{Corollary}

\newcommand{\todo}[1]{\textcolor{blue}{#1}\marginpar{\textcolor{blue}{$\longleftarrow$~To~Do}}}

\theoremstyle{definition}
\newtheorem{definition}[theorem]{Definition}
\newtheorem{example}[theorem]{Example}

\theoremstyle{remark}
\newtheorem{remark}[theorem]{Remark}

\numberwithin{equation}{section}

\newcommand\ignore[1]{}

\newcommand\vare{\varepsilon}

\newcommand\CC{\mathbb C}
\newcommand\HH{\mathbb H}
\newcommand\NN{\mathbb N}
\newcommand\RR{\mathbb R}
\newcommand\ZZ{\mathbb Z}
\newcommand\QQ{\mathbb Q}

\newcommand\Adele{\mathbb A}

\newcommand\vre{\varepsilon}

\newcommand\supp{\operatorname{supp}}

\newcommand\Ad{\operatorname{Ad}}

\newcommand\vol{\operatorname{vol}}

\newcommand\norm[1]{\left\|#1\right\|}

\newcommand\abs[1]{\left|#1\right|}

\newcommand\inn[1]{\left\langle #1 \right\rangle}
\newcommand\set[1]{\left\{{#1}\right\}}
\begin{document}
%\title{The lattice point counting problem with 
%uniform error estimates}
\title{Counting lattice points}

\author{Alexander Gorodnik}
\address{School of Mathematics \\ University of Bristol \\ Bristol BS8 1TW, U.K.}
\email{a.gorodnik@bristol.ac.uk}
\thanks{The first author was supported in part by NSF Grant and RCUK fellowship}

\author{Amos Nevo}
\address{Department of Mathematics, Technion, Haifa, Israel.}

%\curraddr{Department of Mathematics, Technion}
\email{anevo@tx.technion.ac.il}
\thanks{The second author was supported by the IAS at Princeton and  ISF grant 975-05}

%    General info
\subjclass{Primary 22D40; Secondary 22E30, 28D10, 43A10, 43A90}

%\date{November 2007}

\dedicatory{}

\keywords{}

\begin{abstract}
For a locally compact second countable group $G$ and a lattice subgroup $\Gamma$,
we give an explicit quantitative solution of the lattice point counting problem in general domains in $G$, 
provided that   

i) $G$ has finite upper local dimension, and the domains satisfy a basic regularity condition, 

ii) the mean ergodic theorem for the action of $G$ on $G/\Gamma$ holds, with a rate of convergence.

The error term we establish matches the best current result for balls in symmetric spaces of simple higher-rank Lie groups, but holds in much greater generality. 

A significant advantage of the ergodic theoretic approach we use is that the solution to the lattice point counting problem  is uniform over families of lattice subgroups provided they admit a uniform spectral gap. In particular, the uniformity 
property holds for families of finite index subgroups  satisfying a quantitative variant of property $\tau$. 

We discuss a number of applications, including: counting lattice points in general domains in semisimple $S$-algebraic groups, counting rational points on group varieties with respect to a height function, and quantitative angular (or conical) equidistribution of lattice points in symmetric spaces and in affine symmetric varieties. 

We note that the mean ergodic theorems which we establish are based on spectral methods,
including the spectral transfer principle and the Kunze-Stein phenomenon. We formulate and prove appropriate  analogues of both of these results in the set-up of adele groups, and they constitute a necessary step in our proof 
of quantitative results in counting rational points.

%The general principle underlying the solution  
%is that if the Haar-uniform averages supported on the domains satisfy a quantitative mean ergodic theorem in their action  on the homogeneous space $G/\Gamma$, then mild regularity conditions on the domains suffice to imply a solution of the lattice point counting problem in the domains.

%The method presented gives a simple uniform explicit solution 
%to a diverse collection of non-Euclidean lattice point counting problems, on the group variety as well as on certain homogeneous spaces, going beyond the case of affine symmetric varieties. 

%
%In particular, we describe a solution to the problem of counting rational points 
%on a group variety with respect to a height function for general
% algebraic groups, viewing the group of rational points a s a lattice in the Adele group. The method 
%leads to explicit error terms and improves the results 
% obtained recently in the semisimple 
%case in \cite{STBT2} 
%\cite{GMO}.  

\end{abstract}

\maketitle

{\small \tableofcontents}

\section{Introduction, definitions, and statements of general counting results}
\subsection{Introduction and definitions}
Let $G$ be a locally compact second countable (non-compact) group and  $\Gamma$ a discrete
subgroup of $G$ with finite covolume. The purpose of the present paper is to give a general solution to the problem of counting lattice points in families of domains in $G$. More explicitly, our goal is to show that for a family of subsets $B_t\subset G$, $t > 0$,
\begin{equation}\label{eq:bt}
|\Gamma\cap B_t|\sim m_G(B_t)\quad\hbox{as $t\to\infty$},
\end{equation}
where $m_G$ is Haar measure on $G$ normalised by $m_{G/\Gamma}(G/\Gamma)=1$. 
Furthermore, we seek to establish an error term in the asymptotic, of the form 
\begin{equation}\label{eq:er}
\abs{\frac{\abs{ \Gamma\cap B_t}}{m_G (B_t)}-1}\le C\, m_G(B_t)^{-\delta}\,\,.
\end{equation}

Our approach is based on the    
the following fundamental principle : {\it the main term in the number of lattice points follows from 
the mean ergodic theorem in $L^2(G/\Gamma)$ for the Haar-uniform averages supported on the sets $B_t$, and the error estimate follows from the rate of convergence of these averages.}
This principle is a part of the general ergodic theory of lattice subgroups formulated in  \cite{GN} and here we systematically develop and refine the diverse counting results which it implies. 

In general, given a  family of domains $B_t\subset G$ and an ergodic  measure-preserving action of $G$
on a probability measure space $(X,\mu)$, the mean ergodic theorem (for the family $B_t$) 
is the statement that
\begin{equation}\label{eq:met}
\frac{1}{m_G(B_t)}\int_{B_t}f(g^{-1}x) dm_G(g)\stackrel{L^2}{\longrightarrow} \int_{X} f\, d\mu\quad\hbox{as $t\to\infty$}
\end{equation}
for every $f\in L^2(X)$.
%and $m_{G/\Gamma}$ is the normalized $G$-invariant measure on $G/\Gamma$.
We show that the mean ergodic theorem, together with a mild regularity property 
for the sets $B_t$ (namely, well-roundedness \cite{DRS}, \cite{EM}), implies that 
\eqref{eq:bt} holds. Furthermore, when convergence 
takes place with a fixed rate, the sets $B_t$ satisfy a quantitative regularity condition (namely,
H\"older well-roundedness \cite{GN}), and $G$ has finite upper local dimension, then 
the lattice point counting problem for the domains $B_t$ admits an explicit quantitative solution. The error term is
controlled directly by the spectral gap estimate satisfied by the family of averaging operators above 
acting on $L^2(G/\Gamma)$, together with the degree of regularity of $B_t$ and the upper local dimension.

 Previous counting results in the literature are improved upon in several different respects, including admitting more general sets,
 establishing or improving explicit error terms, and enlarging the class of groups involved. 
 Our approach also gives uniform estimates over families of lattice subgroups (as well as over their cosets), which have a number of interesting applications (see Section \ref{s:application}).

%Another natural problem that arises in applications is counting lattice points in homogeneous
%varieties. Namely, assume that $H\subset G$ is a closed subgroup, and that $\Gamma\cap H$ is a lattice
%subgroup of $H$. Then the orbit of the point $H\in G/H$ under $\Gamma$ is a discrete set of points, and
%given a family $B_t$ of domains in $G/H$ one wants to give a quantitative solution of the counting
%problem $\abs{\Gamma H\in B_t}$. This problem encompasses a vast collection of counting problems, many
%of which are completely open. But as we shall see, in many interesting cases the problem can be reduced
%to the problem of counting lattice points in certain subsets $\tilde{B}_t$ in $G$ lying  above $B_t$, and solved using the mean ergodic theorem.  This approach can be used beyond the case of affine symmetric varieties and applies to more general reductive homogeneous spaces, as we will see below.  

%\section{Statement of the main results}
%Now we state the main result reducing the lattice point counting problems in $G$ to mean ergodic
%theorems on $G/\Gamma$. 

We begin by recalling and introducing some definitions needed in the statements of the main results.
Let  $\mathcal{O}_\vre$, $\vre > 0$, be a family  
of symmetric neighbourhoods of the identity in $G$, which is decreasing with $\vre$. 
Let $B_t\subset G$, $t\in \RR_+$,  be a family of bounded Borel subsets 
of positive finite Haar measure.  In the following definition, we recall the notion of well-rounded sets from \cite{DRS} and  \cite{EM}, and give an effective version of it (see \cite{GN}). 

\begin{definition}\label{well rounded}
{\bf Well-rounded and H\"older well-rounded sets.}
\begin{enumerate}
\item 
The family $B_t$ is {\it well-rounded} (w.r.t. $\mathcal{O}_\vre$) if for every $\delta > 0$ there exist $\vare,t_1> 0$  such that for all $t\ge t_1$,
$$ m_G(\mathcal{O}_\vre B_t \mathcal{O}_\vre)\le (1+\delta) m_G(\cap_{u,v\in \mathcal{O}_\vre} uB_t v)\,.$$
\item The family $B_t$ is {\it H\"older well-rounded} with exponent $a$ (w.r.t. $\mathcal{O}_\vre$) if 
there exist $c,\vre_1,t_1>0$ such that for all $0< \vre < \vre_1 $ and $t\ge t_1$, 
\begin{equation*}
 m_G(\mathcal{O}_\vre B_t \mathcal{O}_\vre)\le (1+c \vre^a) m_G(\cap_{u,v\in \mathcal{O}_\vre} uB_t v)\,.
\end{equation*}
\end{enumerate}
\end{definition}

Given a family $B_t$ of subsets of $G$, we set 
\begin{equation}\label{eq:btplus}
B_t^+(\vre)=\mathcal{O}_\vre B_t \mathcal{O}_\vre\,\,\,\hbox{and} \,\,\, B_t^-(\vre)=\cap_{u,v\in
  \mathcal{O}_\vre} uB_t v\,.
\end{equation}

Let us also recall the following natural condition, which is clearly stronger than  H\"older well-roundedness.
\begin{definition}\label{admissible}{\bf Admissible  $1$-parameter families} \cite{GN}.
%\begin{enumerate}
%\item {\it Admissible $1$-parameter families}. 
The family  $B_t$ is {\it H\"older admissible} with exponent $a$ (w.r.t. $\mathcal{O}_\vre$)
 if there exist $c,t_1,\vre_1 > 0$ such that for all
  $0 < \vre< \vre_1$ and  $t\ge t_1$,
\begin{align*}
\mathcal{O}_\vre
\cdot B_t\cdot \mathcal{O}_\vre &\subset B_{t+c\vre^a},\\
m_G(B_{t+\vre})&\le (1+c\vre^a)\cdot  m_G(B_{t}).
\end{align*}

%%\mathcal{O}_{\delta t}\subset G_t\,\,\text{ for some  } \delta > 0. 
%%\label{eq:3}\\
%\item {\it Admissible sequences}. An increasing sequence 
% bounded Borel subset $G_t$, $t\in \NN_+$, 
%on an lcsc totally disconnected group 
% $G$ will be called 
%{\it admissible}  if it is coarsely admissible,  
%and there exists $t_0 > 0$ and 
%a compact open subgroup $K_0$ such that for $t\ge t_0$ 
%\begin{align} K_0 G_t K_0 =G_t\,. \label{eq:4}
%\end{align} 
%\end{enumerate}

\end{definition}

H\"older admissibility (and in some considerations even Lipschitz admissibility) played an important role  in the arguments in \cite{GN} applied to prove pointwise ergodic theorems for general actions of a group $G$ and a lattice subgroup $\Gamma$. For an extensive list of example of admissible averages 
on $S$-algebraic groups we refer to \cite[Ch. 7]{GN}. However, as already noted in \cite{GN}, when we consider only
the mean ergodic theorem on spaces of the form $G/\Gamma$, the condition of H\"older well-roundedness
will be sufficient.  This condition allows for a very diverse set of averages, as we shall see in the
examples below. For instance, the sets arising in the study of angular distribution
of lattice points (see Sections \ref{sec:sym} and \ref{sec:affine}) are H\"older well-rounded,
but not H\"older admissible.

%\begin{definition

%\begin{definition}\label{admissible}{\bf Admissible sets} {\rm
%An increasing family of bounded Borel subset 
%$B_t$, $t>0$, of $G$ will be called 
%{\it admissible} if 
%there exists $c>0$, $\vare_0 > $ and $t_0 > 0$ 
%such that for all $t\ge t_0$ and $0< \vare < \vare_0$ 
%\begin{align}
%\mathcal{O}_\vare\cdot B_t\cdot
% \mathcal{O}_\vre &\subset B_{t+c\vare},\label{admissible:1}\\
%m_G(B_{t+\vre})&\le (1+c\vre)\cdot  m_G(B_{t}).
%\label{admissible:2}
%\end{align}
%}
%\end{definition}

The family of neighbourhoods $\mathcal{O}_\vre$  gives rise to the notion of  upper local dimension:

\begin{definition}{\bf Upper Local dimension.} \label{def:local}
 We say that the {\it upper local dimension}
is at most $\rho$ if there exist $m_0,\vre_1>0$ such that
\begin{equation*}
m_G(\mathcal{O}_\vre)\ge m_0\vre^\rho\quad\hbox{for all $\vre\in (0,\vre_1)$.}
\end{equation*}
\end{definition} 

For example, when $G$ is a connected Lie group, we fix a Riemannian 
metric on $G$ and set
$\mathcal{O}_\vare=\set{g\in G:\, d(g,e)<\vare}$.
Then one can take $\rho=\dim G$.
Another important example is the case where $G=G_\infty \times G_f$ is a product of a connected Lie group of positive  dimension, and a totally disconnected group $G_f$. Set $\mathcal{O}_\vre=
\mathcal{O}_\vre^\infty \times W$, where $\mathcal{O}_\vre^\infty$ are Riemannian  balls of radius 
$\vre$ centered at the identity of $G_\infty$, and $W$ is a fixed compact open subgroup of $G_f$.
Then again $\rho$ is the dimension of $G_\infty$.

Let $\beta_t$ denote the normalised Haar-uniform measure supported on the set $B_t$. 
Consider a measure-preserving action of $G$ on a 
standard Borel probability space $(X,\mu)$ and the averaging operators $\pi_X(\beta_t)$
defined by
\begin{equation}\label{eq:b0}
\pi_X(\beta_t)f(x):=\frac{1}{m_G(B_t)}\int_{B_t}f(g^{-1}x) dm_G(g),\quad f\in L^p(X).
\end{equation}

\begin{definition}\label{quant mean}{\bf Mean ergodic theorems.}
\begin{enumerate}
\item The operators 
$\pi_X(\beta_t)$ satisfy the {\it mean ergodic theorem in $L^2(X)$} if 
$$\norm{\pi_X(\beta_t)f-\int_X f d\mu}_{L^2(X)}\to 0\quad\hbox{as $t\to\infty$}$$ 
for all $f\in L^2(X)$. 
\item The operators 
$\pi_X(\beta_t)$ satisfy
 the {\it quantitative mean ergodic theorem in $L^2(X)$} 
with rate $E(t)$ if 
$$\norm{\pi_X(\beta_t)f-\int_X fd\mu}_{L^2(X)}
\le E(t) \norm{f}_{L^2(X)}$$
for all $f\in L^2(X)$ and $t>0$,
where $E:(0,\infty)\to (0,\infty)$ is a function such that $E(t)\to 0$ as $t\to\infty$.
\end{enumerate}
\end{definition}

Note that when the action of $G$ on $X$ satisfies the quantitative mean ergodic theorem,
the unitary representation of $G$ in $L^2_0(X)$ (the space of mean zero functions)
must have a spectral gap, provided that at least one  of the sets $B_t^{-1}B_t$ generates $G$
for some $t$ such that $E(t)<1$.
Conversely, when $G$ is a connected semisimple Lie group
with finite center and the unitary representation of $G$ in $L^2_0(X)$
has a (strong) spectral gap, for general families $B_t$ one has a quantitative
mean ergodic theorem of the form
\begin{equation*}
\norm{\pi_X(\beta_t)f-\int_X fd\mu}_{L^2(X)}
\le C\, m(B_t)^{-\kappa} \norm{f}_{L^2(X)}
\end{equation*}
for some $C,\kappa>0$ (see Theorem \ref{semisimple mean} below).

In order to solve the lattice point counting problem,
we will need a stable version of the mean ergodic theorem:

\begin{definition}\label{stable}{\bf Stable  mean ergodic theorems}.
We will call the ergodic theorem for averages
along the sets $B_t$ {\it stable} if it holds for all the families 
$B_t^+(\vre)$ and $B_t^-(\vre)$ simultaneously for $\vre\in (0,\vre_1)$. 
For the quantitative mean ergodic theorem we require in addition that the function $E(t)$
is independent of $\vre$. 
\end{definition}
\begin{remark}
As we shall see, in the context of semisimple $S$-algebraic groups a well-rounded family which
satisfies the (quantitative) mean ergodic theorem also satisfies the (quantitative) stable mean ergodic theorem. Indeed, our method of establishing the norm estimate associated with a strong spectral gap  is based on the spectral transfer principle and the Kunze-Stein phenomenon. Together these imply that the rate of convergence depends only on the rate of volume growth of the family, and the $L^p$-parameter of integrability of the representation. We will establish and then apply similar considerations to radial averages on adele groups. 

%{\bf Question : What about general groups ? when is the stable quantitative theorem  a consequence of H\"older well-roundedness and spectral gap  ???? } 

\end{remark}

\subsection{Statement of general counting results}

We can now formulate the following basic result, which provides the main term in the lattice point counting problem in well-rounded domains.  Fixing a choice of Haar measure on 
$G$, let us denote the measure of a fundamental domain of $\Gamma$ in $G$ by $V(\Gamma)$.

 \begin{theorem}\label{main term}{\bf Lattice point problem : main term.}
 Let $G$ be an lcsc group, $\Gamma\subset G$ a discrete lattice subgroup,
and $B_t$ a well-rounded family of subsets of $G$. 
Assume that the averages $\beta_t$ supported on  $B_t$ satisfy the stable mean ergodic theorem in
$L^2(G/\Gamma)$.
Then
 $$\lim_{t\to \infty} \frac{\abs{\Gamma\cap B_t}}{m_G( B_t)}=\frac{1}{V(\Gamma)}\,.$$
\end{theorem}

The argument of the proof of Theorem \ref{main term} also applies to count points in translated cosets of lattice subgroups.

\begin{corollary}\label{cor:shift0}
Under the conditions of Theorem \ref{main term}, 
 $$\lim_{t\to \infty} \frac{\abs{x\Gamma y^{-1}\cap B_t}}{m_G( B_t)}=\frac{1}{V(\Gamma)}$$
for  every $x,y\in G$.
\end{corollary}

\ignore{
\begin{remark}\label{cosets}
We note that under the assumptions of Theorem \ref{main term}, we also have
\begin{equation}\label{eq:g}
\lim_{t\to \infty} \frac{\abs{x\Gamma y^{-1} \cap B_t}}{m_G( B_t)}=\frac{1}{V(\Gamma)}
\end{equation}
for every $x,y\in G$.  
In particular, any left or right coset of $\Gamma$ has the same asymptotic 
count as $\Gamma$. Indeed, to verify \eqref{eq:g} it is sufficient to observe that
the family $x^{-1}B_t y$ is well-rounded w.r.t. to the basis $\mathcal{O}_\vre'=
x^{-1}\mathcal{O}_\vre x\cap y^{-1}\mathcal{O}_\vre y$
of 
neighbourhoods of identity, and the averages supported $x^{-1}B_t y$ satisfy
the stable mean ergodic theorem w.r.t. $\mathcal{O}_\vre'$.

We will also prove a quantitative version of \eqref{eq:g} below
(see Corollaries \ref{cor:shift} and \ref{c:compact}). 
\end{remark} 

\todo{Question is this true in full generality for all groups ??? is $x^{-1}B_t y$  a well rounded family, in general ? }
}

As noted in \S 1.1, to handle the error term we will use a quantitative estimate on the rate of $L^2$-norm convergence of the averages $\pi_{G/\Gamma}(\beta_t) $ to the ergodic mean, together with a quantitative form of well-roundedness. This will give a uniform quantitative solution to the lattice point counting problem. 
Before we formulate this result, we summarise our notation:
\begin{align}\label{eq:not}
c, a & =\hbox{the H\"older well-roundedness parameters of the family $B_t$,}\nonumber\\
m_0, \rho & =\hbox{the local upper dimension estimate for the group $G$,}\\
E(t) & = \hbox{the error estimate in the stable mean ergodic theorem for $B_t$.}\nonumber
\end{align}
For $g\in G$, we set
\begin{align*}
\vre_0(g,\Gamma) &=\sup\{\vre>0:\; \hbox{$\mathcal{O}_{\vre}^2g$ injects in $G/\Gamma$}\}.
\end{align*} 

\begin{theorem}\label{error estimate}{\bf Lattice point problem : error term.}
Let $G$ be an lcsc group, and 
$B_t$ a H\"older well-rounded family of subsets of $G$, w.r.t. a family $\mathcal{O}_\vre$ of upper local dimension at most $\rho$.  
Let $\Gamma \subset G$ be any discrete lattice subgroup, and assume that the averages $\beta_t$ satisfy  the stable quantitative mean ergodic theorem in $L^2(G/\Gamma)$ with 
 rate $E(t)$, and that $\vre_0(e,\Gamma)\ge \vre_0$. 
Then there exists $t_0>0$ such that for $t\ge t_0$,
\begin{equation}\label{eq:AA}
\abs{\frac{\abs{\Gamma\cap B_t}}{m_G( B_t)}- \frac{1}{V(\Gamma)}}\le   A\, E(t)^{a/(\rho+a)}\,\,,
\end{equation}
where $A=(4 m_0^{-1})^{a/(\rho+a)} (c\, m_G(\mathcal{O}_{\vre_0})^{-1})^{\rho/(\rho+a)}$.
\end{theorem}

\begin{remark}\label{rem:main}
Let us note the following regarding Theorem \ref{error estimate}. 
\begin{enumerate}
\item The estimate (\ref{eq:AA}) is independent of the choice of Haar measure $m_G$.
\item If the sets $B_t$ are bi-invariant under a compact subgroup $K$ of $G$, we can take $\rho$ to
be the upper local dimension of $K\backslash G$. Indeed, then the proof of Theorem \ref{error estimate} can then be carried out in the space $L^2(K\backslash G/\Gamma)$.

\item 
%The rate of the error estimate  
%depends on  the decay of $E(t)$ (namely on the spectral gap),
%the local dimension $\rho$ of $G$, and the H\"older parameter $a$ of the family $B_t$. 
%The constant $A$ in the error term depends on the parameters $c$, $a$ in the H\"older well-roundedness property
%of  the family $B_t$, the parameters $m_0$, $\rho$ in the dimension estimate 
%$m_G(\mathcal{O}_\vre)\ge m_0 \vre^\rho$,  and the measure of the set $\mathcal{O}_{\vre_0}$ injecting into 
%$G/\Gamma$. 
The constant $t_0$ depends on all the parameters in (\ref{eq:not}), as well as $\vre_0$, $\vre_1$, and $t_1$
(appearing in Definitions \ref{admissible}, \ref{def:local}, \ref{stable}).
\end{enumerate}
\end{remark}

As to counting points in translated cosets of the lattice, we note the following result which will be shown below to follow from  Theorem \ref{error estimate}.
%Let $\hbox{Inj}(G,\Gamma,\vre_0)$ be the set of $g\in G$ such that $\mathcal{O}^2_{\vre_0}g$ projects injectively on $G/\Gamma$.

\begin{corollary}\label{cor:shift}
Under the conditions of Theorem \ref{error estimate}, there exists $t_0>0$ such that for  $x,y\in G$,
satisfying  $\vre_0(x,\Gamma),\vre_0(y,\Gamma)\ge \vre_0$, and $t\ge t_0$,  
\begin{equation}\label{eq:xy}
\abs{\frac{\abs{x\Gamma y^{-1}\cap B_t}}{m_G( B_t)}- \frac{1}{V(\Gamma)}}\le   A\, E(t)^{a/(\rho+a)}
\end{equation}
where $A$ is the same as in Theorem \ref{error estimate}.
\end{corollary}

\begin{remark}
If $\Omega$ is a subset of $G$ bounded modulo $\Gamma$, then (\ref{eq:xy}) 
holds for all $x,y\in \Omega$ and $t\ge t_0(\hbox{diam}(\Omega\Gamma))$,
where the constant $A$ is given as in Theorem \ref{error estimate} with 
$\vre_0=\inf\{\vre(g,\Gamma):\, g\in \Omega\}$.
In particular, if  the lattice $\Gamma$ is co-compact, (\ref{eq:xy}) 
holds uniformly over {\it all} $x,y\in G$.
\end{remark}

\ignore{
 \todo{Question : did we really prove this in full generality for all groups, not just for semisimple groups ?   Are we claiming that $x^{-1}B_t y$ are Holder well-rounded ?  Also, what if the spectral gap is not strong, even in the semisimple case ? Do we then have a method of passing directly from a rate for $B_t$ to a rate for $B_t(\vare)$, even for semisimple groups, when the spectral gap is not strong ?}
}

 An important consequence of Theorem \ref{error estimate} is that given the family $B_t$, the error estimate depends only on the size of the spectral gap, and on the  size of the largest neighbourhood $\mathcal{O}_\vare^2$ which injects into $G/\Gamma$. In particular, the error estimate holds {\it uniformly}  for all lattice subgroups in $G$ for which these two parameters have fixed lower bounds.   
 
 Specialising Theorem \ref{error estimate} further, we fix the lattice $\Gamma$, and note that then the error estimate holds uniformly over an infinite family of finite index subgroups of $\Gamma$, provided only that the family satisfies a uniform spectral estimate (which is a quantitative version of property $\tau$).  This fact will be exploited in Section 
 \ref{sec:conguence}, where we consider the congruence subgroups of an arithmetic lattice, and has other uses as well. We formulate it separately, as follows. 
% Given a fixed lattice subgroup $\Gamma_0$, the estimate \eqref{eq:xy} holds
%uniformly for all $x,y\in \Gamma_0$ and for all finite index subgroups $\Gamma\subset\Gamma_0$
%satisfying property $\tau$ (namely, provided that the representations
%on $L^2_0(G/\Gamma)$ have uniform spectral gap). 
 
 \begin{corollary}\label{c:compact}
Assume that the conditions of Theorem \ref{error estimate} hold for the lattice $\Gamma_0$ and a family $\Gamma_j$, $j\in \NN$ of its finite-index subgroups. Given a set $\Omega$ which is compact modulo $\Gamma_0$, there exists $t_0>0$ such that for $t\ge t_0$ and  $x,y\in \Omega$, uniformly for all $j\in \NN$
$$
\abs{\frac{\abs{x\Gamma_j y^{-1}\cap B_t}}{m_G( B_t)}- \frac{1}{V(\Gamma_j)}}\le   A\, E(t)^{a/(\rho+a)}
$$
where  $A$ is given as in Theorem \ref{error estimate} with $\vre_0=\inf\{\vre_0(g,\Gamma):\, g\in \Omega\}$. 

%In particular, when $\Gamma$ is a uniform lattice, the estimate holds for all $x,y\in G$, provided 
% $t\ge t_0(\Gamma)$ which depends on the diameter of a compact fundamental domain for $G/\Gamma$.
 \end{corollary}

\ignore{
\todo{Actually here there is dependence of the constant $A$ on $\Omega$ too ! Also, the same problem with cosets}
}

%\begin{remark}
%\begin{enumerate}
%\item 
%The corollary holds, in particular, for uniform lattices and H\"older well-rounded averages on  semisimple $S$-algebraic groups. Another case is that of uniform lattices and radial averages on semisimple Adele groups.
%\item {\bf The corollary does NOT hold in the case for NON-uniform lattice - e.g. $SL_2(\ZZ)$ or whenever there are (unipotent) elements  that can by conjugated to the identity ?? }
%\item   Note that translation by $\gamma\in \Gamma$ does not change the error term, even when $\Gamma$ is non-uniform. 
%\end{enumerate}
%\end{remark}

\subsection{Comparison with the existing literature}

The problem of the asymptotic development of the number of lattice points in Euclidean space has a long history, going  back to Gauss,
who initiated the study of the number of integral points in domains of the Euclidean plane. For Euclidean space, and more generally for spaces with polynomial volume growth,
there is a simple geometric argument to derive the main term in the asymptotic, and an error estimate for sufficiently regular convex sets has been established using the Fourier transform \cite{Hl}\cite{Hz}. For Lie groups of exponential volume growth already the main term, and certainly the error term, require significant analytic techniques. The first non-Euclidean counting result in the 
semisimple case is due to Delsarte \cite{d}. Currently there are several 
different approaches to non-Euclidean lattice point counting problems for $\Gamma\subset G$ in certain cases, as follows.

\begin{enumerate}
\item via direct spectral expansion and regularisation of the automorphic kernel on $G/\Gamma$, $G$ semisimple 
\cite{bt1,bt2,BMW,CLT,d,DRS,go,h,lp,le,p,mw,ST,STBT1,STBT2,se},

\item via mixing of $1$-parameter flows, or equivalently decay of matrix coefficients for semisimple groups \cite{b1,b2,BO,EM,GMO,mar0,mar1, mar,Ma, mu1, mu2},
 
\item via symbolic coding of Anosov flows and transfer operator techniques \cite{la,po,q,sh}, for lattices in simple groups of real rank one, 

\item via the Weil bound for Kloosterman sums \cite{b,hz}, when counting in congruence subgroups 
of $\hbox{SL}_2$,

\item via equidistribution of unipotent flows \cite{EMS,GO,GO2} (this approach does not
provide an error term in the asymptotic).

\end{enumerate}

Our approach, which is different from those listed above, is based on the mean ergodic theorem for the averages $\beta_t$ acting on $G/\Gamma$,
and has a number of advantages: 
\begin{itemize}
\item {\it Simplicity of the method.}
The quantitative mean ergodic theorem as well as the resulting estimates for counting lattice points are  established by  relatively elementary spectral
and geometric comparison arguments which hold in great generality.  They are valid, in particular, for all semisimple $S$-algebraic  groups and semisimple adele groups, but are not restricted to them. The arguments avoid the complications arising from direct spectral expansion of the automorphic kernel on $G/\Gamma$ (in particular, those associated with regularisation of Eisenstein series) which method (1) introduces. As a result,  this approach allows a considerable  expansion of the scope of quantitative lattice point counting results. 
%and the complications related to the estimates on matrix coefficients using Sobolev norms.

\item {\it Quality of error term.} The quality of the error term derived via the quantitative mean ergodic theorem matches or exceeds the currently known
  best bound in all the non-Euclidean lattice point counting problems we are aware of, with the exception
  of \cite{se}, \cite{lp}, \cite{p} and \cite{BMW}.
Note that the latter results deal only with lattices in real rank one Lie groups (or their products), and only with sets which are bi-invariant under 
a maximal compact subgroup. These assumptions make it possible to deploy method (1) using very detailed information regarding the special functions 
appearing in the spherical spectral expansion.  For more general domains, the approach 
via the decay of matrix coefficients in method (2) can be used to give an error estimate in the semisimple case, but its quality is inferior to the one stated above. (For more details  cf. \cite[Ch.~2]{GN} and compare with \cite{Ma} for the case of real groups, and with \cite{GMO} for the case of adele groups). 

\item {\it Uniformity over lattice families.} The quantitative ergodic theorem governing the behaviour of 
the averages $\beta_t$ is valid for all ergodic probability-measure-preserving actions of $G$
satisfying the same spectral bound. In particular it holds uniformly for all the homogeneous probability spaces $G/\Gamma$, as $\Gamma$ ranges over lattice subgroups of $G$, provided they satisfy a fixed spectral gap estimate. Together with an obvious
necessary geometric condition on the fundamental domains, our approach solves the lattice point counting problem uniformly over this class of lattices.  
In particular, the counting result holds uniformly for all the finite index
subgroups of an irreducible lattice $\Gamma$ in a semisimple $S$-algebraic group,  which satisfy a uniform spectral gap property, namely property $\tau$, for example congruence subgroups (see \S  \ref{sec:conguence} and \S 1.4 for some applications of the uniformity property).

\item {\it Generality of the sets}. The mean ergodic theorem is typically very robust, namely it holds in great generality  
and is usually stable under geometric perturbations of the averaging sets. The same two features hold for the quantitative mean ergodic theorem on the homogeneous probability space $G/\Gamma$. 
This allows us to establish the first quantitative counting results for general families of non-radial
sets on semisimple groups, including such natural examples as  bi-sectors on symmetric spaces and affine
symmetric spaces (see Section \ref{sec:sym} and Section \ref{sec:affine}). Note for example that our
error estimate for 
tempered lattices in sectors in the hyperbolic plane matches the one recently obtained 
in \cite{b} via method (4), 
but applies for all lattices (in all semisimple groups) rather than just congruence subgroups of $\hbox{SL}_2(\ZZ)$.  

\item{\it Generality of the groups}.
Mean ergodic theorems originated in amenable ergodic theory and of course do not require the group to be semisimple. The lattice point counting results we present are valid in the case of lattices in amenable groups as well, for example in connected nilpotent or exponential-
solvable Lie groups. Other cases where they can be applied are the affine groups of Euclidean spaces, as well as the associated adele groups. We will treat these matters in more detail elsewhere, but emphasize here that the principle of deriving the solution to the lattice point counting problem from the mean ergodic theorem is completely general and valid for every lcsc group.  
In particular, the assumption of mixing for flows on the space $G/\Gamma$ --- called for in method (2) --- is not relevant to the lattice point counting problem.

\end{itemize}

\subsection{Applications of uniformity in counting}\label{s:application}
Theorem \ref{error estimate} and its corollaries have already found a number of applications beyond those we will describe  below in the present paper. In addition to numerous examples discussed in \cite[Ch.~2]{GN}, let us mention very briefly the following consequences of the uniformity in counting  over congruence subgroups of an arithmetic lattice in an algebraic group. 
%and refer to \cite{GNS} for more details. 

\begin{enumerate}
\item For connected semisimple Lie groups, uniformity was 
first stated and utilised in the problem of sifting  
 the integral points $G(\ZZ)$ on the group variety $G(\RR)$ \cite{NS}. It plays an essential role in  establishing the existence of the right order of magnitude of almost prime points, 
 for example, almost prime integral unimodular matrices. Theorem \ref{uniform} which we formulate below allows the generalisation of these results to $S$-algebraic groups, for example to construct almost prime unimodular $S$-integral matrices. 
% We note that finite-index 
%congruence subgroups assume the role that arithmetic progressions play 
%in the classical sieving problems over the integers, so the uniformity of the error estimate 
%in the lattice point counting problem becomes a crucial ingredient in 
%applying sieving arguments. 
  It is also crucial in establishing the existence of almost prime points 
 on affine symmetric varieties (for example, integral symmetric matrices).

\item Given a affine homogeneous variety ${\sf X}$ of a semisimple algebraic group defined over $\mathbb{Q}$, it is possible to establish using uniformity in counting over congruence groups,  an effective result on lifting of integral points.
Namely, we show that every point in the image of the reduction map
${\sf X}(\mathbb{Z})\to {\sf X}(\mathbb{Z}/p\mathbb{Z})$ can be lifted to a point 
with coefficients of size $O(p^N)$ with fixed $N>0$.
On the other hand, such a result is false for general homogeneous varieties.

\item We obtain an estimate on the number of integral points on certain proper subvarieties ${\sf X}$ of a
group variety ${\sf G}$. Namely, we show that there exists uniform $\alpha\in (0,1)$ such that
$N_T({\sf X})=O_{\sf X}(N_T({\sf G})^\alpha)$, where $N_T(\cdot)$ denotes the number of integral
points with norm bounded by $T$. 
%In a number of cases, this bound is considerably better than the standard Lang-Weil bound for subvarieties. 
\end{enumerate}
The proofs of these results will be given in a separate paper.

\subsection{Organisation of the paper}
We prove the results stated in \S 1.2  in Section  \ref{sec:proof}.
 In Section \ref{sec:ex} we formulate a general recipe for counting lattice points
and  explain the case of sectors in the hyperbolic plane as a motivating example.
%Then we set-up notation in Section \ref{sec:not}.
In the rest of the paper, we discuss several applications of our results.
We discuss lattices in semisimple $S$-algebraic groups (Section \ref{sec:s-arith}),
uniformity over congruence
subgroups and the density hypothesis (Section \ref{sec:conguence}),
rational points on group varieties (Section \ref{sec:rat}), 
and angular distribution of lattice points on symmetric spaces (Section \ref{sec:sym})
and on affine symmetric varieties (Section \ref{sec:affine}).

\section{Ergodic theorems and counting lattice points}\label{sec:proof}

In this section, we prove the results stated in the introduction.
Let $G$ be an lcsc group and $\Gamma$ a discrete lattice subgroup of $G$.
We denote by $\tilde{m}_{G/\Gamma}$ the measure induced on $G/\Gamma$
by our fixed choice of the Haar measure $m_G$  on $G$. 
Thus, $\tilde{m}_{G/\Gamma}(G/\Gamma)$ is the total measure  
of the locally symmetric space $G/\Gamma$, a quantity we denote by $V(\Gamma)$.
We also let $m_{G/\Gamma}$ denote 
the corresponding probability measure on $G/\Gamma$, namely 
$\tilde{m}_{G/\Gamma}/V(\Gamma)$.
Let $\mathcal{O}_\vre$ be a family of symmetric neighbourhoods of the identity and
$$
\chi_\vre=\frac{\chi_{\mathcal{O}_\vre}}{m_G(\mathcal{O}_\vre)}\,.
$$
We consider the function:
$$
\phi_\vre(g\Gamma)=\sum_{\gamma\in \Gamma} \chi_{\vre}(g\gamma).
$$
Note that $\phi_\vre$ is a measurable bounded 
function on $G/\Gamma$ with compact support, and 
\begin{equation}\label{eq:1}
\int_G\chi_\vre\,dm_G=1,\quad
\int_{G/\Gamma}\phi_\vre\,d\tilde{m}_{G/\Gamma}=1, 
\quad
\int_{G/\Gamma}\phi_\vre\,dm_{G/\Gamma}
=\frac{1}{V(\Gamma)}\,.
\end{equation}

Let us now note the following basic observations, which will allow us to reduce the 
lattice point counting problem to the ergodic theorem on $G/\Gamma$, together with a regularity property of the domains.

First, for any $\delta > 0$, $h\in G$ and $t>0$, 
the following are obviously equivalent, by definition, for any family of 
Haar-uniform measures $b_t$ on $G$,
\begin{equation}\label{eq:G_Gamma-1}
\abs{\pi_{G/\Gamma}(b_t)\phi_\vre(h\Gamma) 
-\frac{1}{V(\Gamma)}}\le\delta, 
\end{equation}
\begin{equation}\label{eq:G_Gamma-2}
\frac{1}{V(\Gamma)}-\delta \le \frac{1}{m_G(\text{supp }b_t)}
\int_{\text{supp } b_t}\phi_\vre(g^{-1}h\Gamma)dm_G(g)\le 
\frac{1}{V(\Gamma)}+\delta\,\,.
\end{equation}

We will estimate the first expression using the mean ergodic theorem and Chebycheff's inequality.  
On the other hand, the integral in the 
second expression is connected to lattice points as follows.

\begin{lemma}\label{comparison}
Let $B_t$ be a family of measurable subsets of $G$. 
Then for every $t>0$, $\vre>0$ and $h\in \mathcal{O}_\vre$,
$$
\int_{B_t^-(\vre)}\phi_\vre(g^{-1}h\Gamma)
\,dm_G(g)\le \abs{B_t \cap \Gamma}\le \int_{B_t^+(\vre)}
\phi_\vre(g^{-1}h\Gamma)\,dm_G(g)
$$
where $B_t^+(\vre)$ and $B_t^-(\vre)$ are defined as in \eqref{eq:btplus}.
\end{lemma}

\begin{proof}
If $\chi_\vre(g^{-1}h\gamma)\ne 0$ 
for some $g\in B_t^-(\vre)$, $h\in\mathcal{O}_\vre$, $\gamma\in \Gamma$, then we obtain 
$$
\gamma\in h^{-1}\cdot B_t^-(\vre)
\cdot (\supp\, \chi_\vre)\subset B_t\
$$
since $\mathcal{O}_\vre B_t^-(\vre)\mathcal{O}_\vre \subset B_t$.
Hence, by the definition of $\phi_\vre$ and \eqref{eq:1},
$$
\int_{B_t^-(\vre)}\phi_\vre(g^{-1}h\Gamma)\,dm_G(g)
= \sum_{\gamma\in B_t\cap\Gamma} \int_{B_t}
 \chi_\vre(g^{-1}h\gamma)\,dm_G(g)\le \abs{B_t\cap\Gamma}.
$$

In the other direction, 
for $\gamma\in B_t\cap\Gamma$ and $h\in\mathcal{O}_\vre$,
$$
\supp(g\mapsto \chi_\vre(g^{-1}h\gamma))
=h\gamma(\supp\,\chi_\vre)^{-1}\subset B_t^+(\vre).
$$
Since $\chi_\vre\ge 0$ and \eqref{eq:1} holds,
$$
\int_{B^+_t(\vre)}\phi_\vre(g^{-1}h\Gamma)\,dm_G(g)\ge 
\sum_{\gamma\in B_t\cap\Gamma}
 \int_{B^+_t(\vre)}\chi_\vre(g^{-1}h\gamma)\,dm_G(g)= \abs{B_t\cap\Gamma},
$$
as required.
\end{proof}

\begin{proof}[Proof of Theorem \ref{main term}]
We use a small parameters $\delta>0$. Since the family $B_t$ is well-rounded,
there exists $\vre>0$ such that
\begin{equation}\label{eq:bbbb}
m_G(B_t^+(\vre))\le (1+\delta)m_G(B_t^-(\vre))
\end{equation}
for all sufficiently large $t$.
%Moreover, taking smaller $\vre>0$ if needed, we may assume that $\mathcal{O}_\vre^2$
%injects on $G/\Gamma$. Then $\mathcal{O}_\vre\gamma_1\cap \mathcal{O}_\vre\gamma_2=\emptyset$
%for $\gamma_1\ne \gamma_2\in\Gamma$. In particular, this implies that $\|\phi_\vre\|_\infty\le 1$.

By the stable mean ergodic theorem and \eqref{eq:1},
$$
\norm{\frac{1}{m_G(B_t^+(\vre))}\int_{B^+_{t}(\vre)}
\phi_\vre(g^{-1}h\Gamma)\,dm_G(g) - \frac{1}{V(\Gamma)}}_{L^2(G/\Gamma)}
\to 0$$
and so 
$$
m_{G/\Gamma}\left(\set{h\Gamma\,:\, \abs{
\frac{1}{m_G(B_t^+(\vre))}\int_{B^+_{t}(\vre)}
\phi_\vre(g^{-1}h\Gamma)\,dm_G(g) - \frac{1}{V(\Gamma)}} > \delta}\right)\to 0
$$
as $t\to\infty$. Hence, the measure of this set will be less than $m_G(\mathcal{O}_\vre\Gamma)$
for large $t$. Then there exists $h_t\in\mathcal{O}_\vre$ such that
$$
\abs{
\frac{1}{m_G(B^+_t(\vre))}\int_{B^+_{t}(\vre)}
\phi_\vre(g^{-1}h_t\Gamma)\,dm_G(g) - \frac{1}{V(\Gamma)}} \le  \delta.
$$
Combining this estimate with Lemma \ref{comparison} and \eqref{eq:bbbb}, we obtain that
$$
|\Gamma\cap B_t|\le \left( \frac{1}{V(\Gamma)} +\delta\right)m_G(B^+_{t}(\vre))\le  \left(\frac{1}{V(\Gamma)}+\delta\right)(1+\delta) m_G(B_{t})
$$
for all $\delta>0$ and $t\ge t_1(\delta)$.
Since one can similarly prove the lower estimate, this completes the proof.
\end{proof}

\begin{proof}[Proof of Corollary \ref{cor:shift0}]
We apply both parts of the proof of Theorem \ref{main term} to the function in $L^2(G/\Gamma)$ given by 
\begin{equation}\label{eq:phiy}
\phi_\vre^y(gx\Gamma)=\sum_{\gamma\in \Gamma} \chi_{\vre}(gx \gamma y^{-1}).
\end{equation}
\end{proof}

\begin{proof}[Proof of Theorem \ref{error estimate}]
Recall that we are now assuming that the family $\{B_t\}$ is H\"older well-rounded:
$$
m_G(B_t^+(\vre))\le (1+c\vre^a)m_G(B_t^-(\vre))\quad\hbox{for every $t>t_1$ and $\vre\in (0,\vre_1)$,}
$$
the neighbourhoods $\mathcal{O}_\vre$ satisfy
\begin{equation}\label{eq:dim}
m_G(\mathcal{O}_\vre)\ge m_0\vre^\rho \quad\hbox{for some $m_0>0$ and every $\vre\in (0,\vre_1)$,}
\end{equation}
and the averages along the sets $B_t^\pm(\vre)$, $0<\vre<\vre_1$,
satisfy the stable quantitative mean ergodic 
theorem  with the error term $E(t)$. 

In the proof, we use a parameter $\vre>0$ satisfying
\begin{equation}\label{eq:vre}
\vre<\min\{\vre_0,\vre_1,c^{-1/a}\}
\end{equation}
where $\vre_0$ is such that the projection $\mathcal{O}_{\vre_0}^2\to \mathcal{O}_{\vre_0}^2\Gamma$
is injective. Since the neighbourhoods $\mathcal{O}_\vre$ are symmetric,
\begin{equation}\label{eq:m0}
\mathcal{O}_{\vre_0}\gamma_1\cap \mathcal{O}_{\vre_0}\gamma_2=\emptyset\quad
\hbox{for $\gamma_1\ne \gamma_2\in\Gamma$,}
\end{equation}
and
\begin{equation}\label{eq:m1}
m_{G/\Gamma}(\mathcal{O}_\vre\Gamma)=\frac{m_G(\mathcal{O}_\vre)}{V(\Gamma)}.
\end{equation}
Also, we have
\begin{equation}\label{eq:m2}
V(\Gamma)\ge m(\mathcal{O}_{\vre_0}).
\end{equation}

For any family of Haar-uniform averages $b_t$ satisfying the quantitative 
mean ergodic theorem with the error term $E(t)$
for its action on the probability space 
$(G/\Gamma,m_{G/\Gamma})$, we have for all $t > 0$,
$$\norm{\pi_{G/\Gamma}(b_t)\phi_\vre -
\int_{G/\Gamma}\phi_\vre\,dm_{G/\Gamma}}_{L^2(G/\Gamma)}
\le   E(t)\norm{\phi_\vre}_{L^2(G/\Gamma)},$$
and thus for all $\delta,t >0 $, 
$$m_{G/\Gamma}\left(\set{h\Gamma\,:\, \abs{\pi_{G/\Gamma}
(b_t)\phi_\vre(h\Gamma) -\frac{1}{V(\Gamma)}} > \delta}\right)\le 
\delta^{-2} E(t)^2
\norm{\phi_\vre}^2_{L^2(G/\Gamma)}.$$
It follows from \eqref{eq:m0} that
\begin{align*}
\norm{\phi_\vre}^2_{L^2({G/\Gamma})}&=\int_{G/\Gamma}
\phi_\vre(h\Gamma)^2 \frac{d\tilde{m}_{G/\Gamma}(h\Gamma)}{V(\Gamma)}\\
&= \int_G \chi_\vre^2(g)\frac{dm_G(g)}{V(\Gamma)}=
\frac{m_G(\mathcal{O}_\vre)^{-1}}{V(\Gamma)}.
\end{align*}
Hence,
\begin{equation}\label{eq:measure}
m_{G/\Gamma}\left(\set{h\Gamma\,:\, \left|\pi_{G/\Gamma}
(b_t)\phi_\vre(h\Gamma) -\frac{1}{V(\Gamma)}\right| > \delta}\right)\le 
\frac{m_G(\mathcal{O}_\vre)^{-1}}{ V(\Gamma)}
  \delta^{-2}  E(t)^2\,\,.
\end{equation}
This shows that the measure of the latter set 
decays with $t$.
In particular, the measure will eventually 
be strictly smaller than 
$m_{G/\Gamma}(\mathcal{O}_\vre\Gamma)=m_G(\mathcal{O}_\vre)/V(\Gamma)$ for sufficiently large $t$.
%\begin{remark}
%Note that without the assumption
% that $\mathcal{O}_\vre \gamma$ are disjoint, 
%$\phi_\vre$ is not a multiple of 
%the characteristic function of the projection 
%of $\mathcal{O}_\vre$ to $G/\Gamma$. The same analysis can still 
%proceed, by introducing a constant 
%that depends on  $\Gamma$. 
%\end{remark} 
Then
\begin{equation}\label{eq:intersection}
\mathcal{O}_\vre\Gamma\cap 
 \set{h\Gamma\,:\, \abs{\pi_{G/\Gamma}
(b_t)\phi_\vre(h\Gamma) -\frac{1}{V(\Gamma)}} \le \delta}
\neq \emptyset.
\end{equation}
Thus according to (\ref{eq:G_Gamma-1})
 and (\ref{eq:G_Gamma-2}) applied to the sets $B_t^+(\vre)$, 
for any $h$ in the non-empty intersection 
 (\ref{eq:intersection}),
\begin{equation}
 \frac{1}{m_G(B_t^+(\vre))}
\int_{B_t^+(\vre)}\phi_\vre(g^{-1}h\Gamma)dm_G(g)\le 
\frac{1}{V(\Gamma)}+\delta.
\end{equation}
On the other hand, by Lemma \ref{comparison},  for $h\in \mathcal{O}_\vre$, 
\begin{equation}\label{eq:admissibility}
\abs{\Gamma\cap B_t}\le 
\int_{B_t^+(\vre)}\phi_\vre(g^{-1}h\Gamma)dm_G(g)\,\,.
\end{equation}
Combining these estimates and using the fact that the family $\{B_t\}$
is H\"older well-rounded, we conclude that 
\begin{equation}\label{eq:est}
\abs{\Gamma\cap B_t}\le
 \left(\frac{1}{V(\Gamma)}+\delta\right)m_G(B_t^+(\vre))\le 
\left(\frac{1}{V(\Gamma)}+\delta\right)(1+c\vre^a)m_G(B_t).
\end{equation} 
This inequality holds as soon as  (\ref{eq:intersection}) holds, 
and so certainly if we have 
\begin{equation}\label{eq:neq}
\frac{m_G(\mathcal{O}_\vre)^{-1}}{V(\Gamma)}
 \delta^{-2}  E(t)^2\le 
\frac{1}{4} \cdot\frac{m_G(\mathcal{O}_\vre)}{V(\Gamma)}.
\end{equation}
Indeed, then the right hand side  is strictly 
smaller than $m_{G/\Gamma}
(\mathcal{O}_\vre\Gamma)$, so that the intersection 
(\ref{eq:intersection}) is necessarily non-empty. 
We set $\delta=2 m_G(\mathcal{O}_\vre)^{-1} E(t)$ so that the equality in
(\ref{eq:neq}) holds.  Now using (\ref{eq:est}), the estimate 
$c\vre^a<1$, \eqref{eq:m2}, and (\ref{eq:dim}),
we deduce that
\begin{align*}
\frac{\abs{\Gamma\cap B_t}}{m_G(B_t)}-\frac{1}{V(\Gamma)} &\le
 2\delta+ \frac{c\vre^a}{V(\Gamma)}\le 4 m_G(\mathcal{O}_\vre)^{-1} E(t)+
\frac{c\vre^a}{m_G(\mathcal{O}_{\vre_0})}\\
&\le 4 m_0^{-1}\vre^{-\rho} E(t)+
\frac{c\vre^a}{m_G(\mathcal{O}_{\vre_0})}.
\end{align*}
To optimise the error term, we choose
$$
\vre=\left(4 m_0^{-1}c^{-1}m_G(\mathcal{O}_{\vre_0}) E(t)\right)^{1/(\rho+a)}.
$$
Note that since $E(t)\to 0$ as $t\to\infty$, there exists $t_0>0$ such that $\vre$ satisfies
(\ref{eq:vre}) for all $t\ge t_0$. Finally, we obtain that for $t\ge t_0$,
$$
\frac{\abs{\Gamma\cap B_t}}{m_G(B_t)}-\frac{1}{V(\Gamma)}\le A\, E(t)^{a/(\rho+a)}
$$
where
$A=(4 m_0^{-1})^{a/(\rho+a)} (c m_G(\mathcal{O}_{\vre_0})^{-1})^{\rho/(\rho+a)}.$

Note that both the comparison 
argument in Lemma \ref{comparison}, as well 
the estimate (\ref{eq:G_Gamma-2}) 
derived from the mean ergodic theorem,  
give a lower bound in addition to the foregoing 
upper bound. Thus the same arguments 
can be repeated to yield also a 
lower bound for the lattice points count.
This completes the proof of Theorem \ref{error estimate}.
\end{proof}

\begin{proof}[Proof of Corollary \ref{cor:shift}]

As in Lemma \ref{comparison}, the quantity $|B_t\cap x\Gamma y^{-1}|$ can be estimated by integrating 
 the function $\phi_\vre^y$, defined in (\ref{eq:phiy}), on small perturbations of $B_t$. Indeed, for  $\vre<\vre_0(y,\Gamma)$ 
we have  
\begin{align*}
&\mathcal{O}_\vre y\gamma_1\cap \mathcal{O}_\vre y\gamma_2=\emptyset
%\quad \hbox{and}\quad \mathcal{O}_\vre y\gamma_1\cap \mathcal{O}_\vre y\gamma_2=\emptyset
\quad\quad \hbox{for $\gamma_1\ne \gamma_2\in\Gamma$.}
\end{align*}
Then the supports of the functions $g\mapsto\chi_{\vre}(g x\gamma y^{-1})$, $\gamma\in\Gamma$, are
disjoint, and we deduce that
\begin{align*}
\norm{\phi_\vre^y}^2_{L^2({G/\Gamma})}=\frac{m_G(\mathcal{O}_\vre)^{-1}}{V(\Gamma)}
\end{align*}
as before. Also, 
$$
m_{G/\Gamma}(\mathcal{O}_\vre x\Gamma)=\frac{m_{G}(\mathcal{O}_\vre)}{V(\Gamma)}\quad\hbox{and}\quad
V(\Gamma)\ge m_{G}(\mathcal{O}_{\vre_0}).
$$
Using this estimates, the proof proceeds exactly as in Theorem \ref{error estimate}. 
\end{proof}

\begin{proof}[Proof of Corollary \ref{c:compact}]
Clearly, when all $\Gamma_j$ are subgroups of a fixed lattice $\Gamma_0$, if $\mathcal{O}^2_{\vre_0}$
injects into $G/\Gamma_0$, it also injects into $G/\Gamma_j$. 
Since we assume that the operators $\pi_{G/\Gamma_j}(\beta_t)$ satisfy the stable quantitative mean
ergodic theorem, with the same rate $E(t)$ for all $j$, 
%on all $G/\Gamma_j$
%Assuming the family satisfies property $\tau$  namely admits a spectral gap of a fixed size, the rate of decay $E(t)$ of the operator norms of 
%$\beta_t$ can be taken as independent of $j$. 
%These are the only two parameters that play a role in Theorem  \ref{error estimate}, so 
the result follows. 
\end{proof}

\section{Lattice point counting problems : general recipe and an example} \label{sec:ex}
\subsection{General recipe}
In the following sections we will give several applications of the general 
lattice point counting result, namely Theorem \ref{error estimate}. These applications are based on the following  recipe:
if $G$ is an lcsc group $G$, $\Gamma$ a discrete lattice in $G$, and $B_t$ a family of sets 
for which we wish to find the asymptotic of the 
number of lattice points together with 
an error term, Theorem \ref{error estimate} reduces 
the problem to the following two steps:
\begin{enumerate}
\item Establish that 
$\norm{\pi_{G/\Gamma}( \beta_t)f-\int_{G/\Gamma}fdm_{G/\Gamma}}_2\le 
E(t)\norm{f}_2$ for some decaying function $E(t)$,
where $ \beta_t$ denotes the Haar-uniform averages supported on $B_t$ (or ``small perturbations'' thereof).  
%and $\pi^0_{G/\Gamma}(\tilde\beta_t)$ is the corresponding averaging operator
%on the space $L^2_0(G/\Gamma)$, the space of square-integrable functions with zero integral.

\item Establish that the family of sets $B_t$ is H\"older well-rounded  w.r.t. to a local 
neighbourhood family $\mathcal{O}_\vre$ which has finite upper local dimension. 
\end{enumerate} 
%Find the main term in the asymptotic of the volume growth 
%of the sets $B_t$, e.g. $m_G(B_t)\cong Ct^\beta e^{\alpha t}$. 
%Interpret the constants $\alpha,\beta$ and $C$ in terms of the group variety 
%$G$, the lattice $\Gamma$ and the geometry of the sets $B_t$. 

The first step is of spectral nature and requires some information  
regarding the unitary representation theory of $G$, and more 
specifically, the spectrum of $L^2(G/\Gamma)$. 
The second step is geometric, and involves the structure of a neighbourhood family 
$\mathcal{O}_\vre$ in $G$ and the regularity of the sets $B_t$ under small perturbations.

In order to carry out the first step, it will be convenient to use the notion of 
an $L^p$-representation.
\begin{definition}\label{lp rep}{\bf $L^p$-representation}. 
A unitary representation $\pi:G\to U(\mathcal{H})$ of an lcsc group $G$
is called $L^p$ if for vectors $v$, $w$ in some dense subspace of $\mathcal{H}$, the matrix coefficient
$\left<\pi(g)v,w\right>$ is in $L^p(G)$. We also say that the representation is $L^{p+}$
if the above matrix coefficients are in $L^{p+\vre}(G)$ for every $\vre>0$. The least $p$ with this property is denoted by $p^+(\pi)$. 
\end{definition}
The following parameter will be used to control the rate of decay in the mean ergodic theorem
and in the asymptotic of the number of lattice points:
\begin{align}\label{eq:kappa}
n_e(p) &= \left\{
\begin{tabular}{l}
\hbox{ the least even integer greater than or equal to $p/2$, if $p>2$,}\\
\hbox{ 1, if $p=2$}.
\end{tabular} \right.
\end{align}

\subsection{A motivating example : lattice points in plane sectors}
To illustrate the two ingredients of our approach geometrically, let us consider the example of $G=\hbox{SL}_2(\RR)$ acting by isometries on the
hyperbolic plane $\mathbb{H}^2$ (of constant curvature $-1$) and a lattice $\Gamma$ in $G$.
Fix a point $o\in \HH^2$ and consider a sector with vertex $o$, namely,
the region between two infinite geodesic
rays starting at $o$ at an angle $\psi > 0$. We let $S_t(\psi)$ denote the  sector intersected with the disc of (hyperbolic) radius $t$ centered at $o$, and proceed to verify the two conditions stated in the recipe. 

 {\it 1) Regularity under perturbation, and geometric comparison argument}. 
It is evident that the uniform probability measure $\sigma_t(\psi)$ supported on $S_t(\psi)$ is dominated
by $\frac{2\pi}{\psi}\beta^\prime_t$, where $\beta^\prime_t$ is the normalised uniform (hyperbolic) measure 
on a disc of radius $t$ with center $o$. 

Consider now a Cartan polar coordinate decomposition,  $G=KA^+K$, where
$K=\hbox{SO}_2(\RR)=\set{k_\phi\,:\, 0\le \phi < 2\pi}$ and $A^+=\{\hbox{diag}(e^{s/2},e^{-s/2}):\,
s\ge 0\}$.
Then the sets $S_t(\psi)$ are given by the coordinates 
$$S_t(\psi)=\set{k_\phi a_s k_{\phi^\prime}\,:\, 0\le \phi < 2\pi,\, 0\le s < t,\, 0\le \phi^\prime \le \psi}$$ 
and $S_t(\psi)$ are indeed Lipschitz well-rounded. To see that, first recall the (hyperbolic) cosine formula for triangles in the hyperbolic plane 
$$\cosh c=\cosh a\cosh b -\sinh a \sinh b \cos \phi\,.$$
In terms of the Cartan polar coordinates decomposition this formula translates to the fact that $a_tk_\phi a_{s}=k_1a_r k_2$ has Cartan component $a_r$ where $\cosh r =\cosh a\cosh b +\sinh a \sinh b \cos \phi$
 (see e.g. \cite[\S 2.2]{N4}). This immediately implies Lipschitz control of the radial part $a_r$ in the Cartan decomposition under small perturbations. For the angular part in the Cartan decomposition one need only consider the representation on $\RR^2$ and estimate ($e_1$, $e_2$ 
being the standard basis)
$$0 < \inn{a_tk_\phi a_{s}e_1,e_1}=e^{t/2}e^{s/2}\inn{k_\phi e_1,e_1}=\inn{a_r k_2 e_1,k_1^{-1}e_1}$$
$$\le \norm{a_r k_2 e_1}\le \norm{a_r}=\norm{a_tk_\phi a_s}\le e^{t/2} e^{s/2}\,\,.$$
Hence if $\inn{k_\phi e_1,e_1}\ge 1-\epsilon$ then the norm of the vector $a_r k_2e_1$ is at least $ (1-\epsilon)\norm{a_r}$ so that $k_2e_1$ must be close to $e_1$. It follows that 
$\inn{k_2e_1,e_1}\ge (1-C\epsilon)$, and similarly for $k_1$.  Hence the angles of rotation defining the
Cartan components $k_1$ and $k_2$ are close to zero, and the Cartan components depend in a Lipschitz
fashion on the perturbation (see Proposition \ref{wave front} for a general argument). 

{\it 2) Spectral estimate}. 
It is well known that for any lattice $\Gamma$ in $G=\hbox{SL}_2(\RR)$, $\pi_{G/\Gamma}^0$ is an $L^{p+}$-representation, where $p^+=p^+(\Gamma)$. It follows from the spectral transfer principle \cite{N1}(see Section \ref{sec:s-arith} below for a full discussion of the following arguments) that  
\begin{align*}
\norm{\pi_{G/\Gamma}^0(\sigma_t(\psi))} &\le
\left(\frac{2\pi}{\psi}\norm{\lambda_G(\beta_t)}\right)^{1/n_e(p)}=\left(\frac{2\pi}{\psi}\Xi_G(t)\right)^{1/n_e(p)}\\
&=O_\eta\left(\psi^{-1/n_e(p)}e^{-((2n_e(p))^{-1}-\eta)t}\right),\quad \eta>0,
\end{align*}
where $\lambda_G$ denotes the regular representation, and $\Xi_G$ is the Harish-Chandra function
(see e.g. \cite[Section 3.1 and Theorem 3.2.1]{HT}). The same argument shows that $S_t(\psi)$ satisfy the stable quantitative mean ergodic theorem, so that Theorem \ref{error estimate} applies and produces 
the error term stated there. 

%\todo{Commen that this matches  Boca's result in terms of the error quality and explain the dependence on the base point that he mentions ???? }

More generally, both the stable quantitative mean ergodic theorem and Lipschitz well-roundedness hold for spherical caps in hyperbolic spaces of arbitrary dimension and lead to the following result. 

\begin{theorem}\label{hyp lat}{\bf Counting points in sectors in hyperbolic space.}
Let $\HH^m$ denote hyperbolic $m$-space (of constant curvature $-1$) and $S_t(\psi)$ a spherical cap with cone angle $\psi$ (namely intercepting a fraction given by $\psi$ of the area of the unit sphere).
Let $\Gamma$ be any lattice subgroup in $G=\hbox{\rm SO}^0(m,1)$
such that $\pi_{G/\Gamma}^0$ is an $L^{p+}$-representation. 
Then the number of lattice points in the spherical cap obeys
\begin{align*}
\abs{ \{\gamma\in\Gamma:\, \gamma o\in S_t(\psi)\} }=&\frac{v_m}{\vol(\Gamma\backslash \HH^m)}\psi
e^{(m-1)t}\\
&+O_\eta\left(\psi^{-1/n_e(p)}
e^{(m-1)\left(1-\frac{n_e(p)^{-1}}{m(m+1)+2}+\eta\right)t}\right),\quad\eta>0,
\end{align*}
with $v_m>0$ depending only on the dimension $m$, and the implied constant depending only 
on $m$, $p$, $\vre_0(e,\Gamma)$. 
(We assume here that only the identity in $\Gamma$ stabilises $o$, otherwise the main term should be
divided by the size of the stabiliser). 

\end{theorem}

We note that the sectors $S_t(\psi)$ constitute a Lipschitz well-rounded family but it is obviously not an  admissible family according to the Definition \ref{admissible}, so that the counting results in \cite{GN} do not directly apply.  

We refer to \cite{b,mar,ni,sh} for other results on the angular distribution of lattice points
in hyperbolic spaces.  In particular, in the special case when  $m=2$ and $\Gamma$ is a principal 
congruence subgroup of $SL_2(\ZZ)$, an error term in this counting problem was derived by Boca in \cite{b}, who has also raised the question of which lattices in $SL_2(\RR)$ satisfy a similar property. Note that Theorem \ref{hyp lat} applies to general lattice subgroups in hyperbolic spaces, and for tempered lattices $\Gamma\subset SL_2(\ZZ)$ (i.e., when $\pi_{G/\Gamma}^0$ is
$L^{2+}$), the error estimate coincides with the one obtained in \cite{b}.

\section{Lattice points on semisimple $S$-algebraic groups}\label{sec:s-arith}

\subsection{Notation}\label{sec:not}

Let us now set the following notation regarding local fields, algebraic groups and adeles, which will be 
used throughout the rest of the paper.

 Given an algebraic number field $F$, 
we denote by $V$ the set of equivalence classes of valuations of $F$.
The set $V$ is the disjoint union 
$V=V_f\coprod V_\infty$ of the set $V_f$ consisting 
of non-Archimedean valuations
and the set $V_\infty$ consisting of Archimedean 
valuations. More generally, for $S\subset V$, we also have
the decomposition $S=S_f\coprod S_\infty$.
For any place $v\in V$, let $F_v$ denote the completion 
of $F$ w.r.t.  the valuation $v$. Let $\mathcal{O}$ denote the ring of 
integers in $F$, and for finite $v$,
let $\mathcal{O}_v$ be its completion, namely 
$\mathcal{O}_v=\set{x\in F_v\,:\, v(x)\ge 0}$.
We denote by $\mathfrak{p}_v$ the
maximal ideal in $\mathcal{O}_v$, by $f_v=\mathcal{O}_v/\mathfrak{p}_v$ the residue field,
and set $q_v=|f_v|$. As usual, the valuation is normalized by $\abs{s_v}_v=\frac{1}{q_v}$, $s_v$ a uniformizer of $\mathcal{O}_v$. 

We introduce local heights $H_v$. For an Archimedean local field $F_v$, and for $x=(x_1,\dots,x_d)\in F_v^d$ we set 
\begin{equation}\label{eq:H1}
H_v(x)=\left(|x_1|^2_v+\cdots + |x_d|^2_v \right)^{1/2},
\end{equation}
and for a non-Archimedean local field $F_v$, 
\begin{equation}\label{eq:H2}
H_v(x)=\max\{|x_1|_v,\ldots,|x_d|_v\}.
\end{equation}

Let $F_v$, $v\in S$, be a finite family of (nondiscrete) local fields,
and $G_v$, $v\in S$, be the $F_v$-points of a semisimple algebraic
group ${\sf G}_v$ defined over $F_v$. Let $\Gamma$ be a lattice in the group $G=\prod_{v\in S} G_v$.
We fix representations $\rho_v:G_v\to \hbox{GL}_{m_v}(F_v)$, $v\in S$, with finite kernels and index the lattice points according to the height $H$ defined by
\begin{equation}\label{eq:H}
H(g)=\prod_{v\in S} H_v(\rho_v(g_v)),\quad g=(g_v)\in G,
\end{equation}
where $H_v$'s are the local heights defined above.
We set 
\begin{align*}
B_T=\{g\in G:\, H(g)\le T\}.
\end{align*}
 
 \subsection{Counting $S$-integral points}

We can now state our solution to the problem of counting $S$-arithmetic lattice points in $S$-algebraic groups (and more general lattices in the product).  
%\todo{Do we need simply connected? Veca assumes simply connected, but Cowling don't.}

\begin{theorem}\label{th:S-arith}{\bf Counting lattice points in $S$-algebraic groups: height balls.}
Keeping the notation in the previous subsection, assume that the groups ${\sf G}_v$ are simply connected, and at least one of ${\sf G}_v$'s
is isotropic over $F_v$ (or equivalently, $G$ is noncompact).
Fix $\vre_0>0$.
Then there exists $T_0>0$ such that for every lattice $\Gamma$ in $G$
for which the representation $L_0^2(G/\Gamma)$ is $L^{p+}$, $x,y\in G$
such that $\vre_0(x,\Gamma),\vre_0(y,\Gamma)\ge \vre_0$, and $T\ge T_0$,
\begin{align*}
|x\Gamma y^{-1}\cap
B_T|&=\frac{\vol(B_T)}{\vol(G/\Gamma)}+O_\eta\left(\vre_0^{-d^2/(a+d)}\,\vol(B_T)^{1-(2n_e(p))^{-1}
      a/(a+d)+\eta}\right)
%&=\frac{T^\alpha P(\log T)}{\vol(G/\Gamma)}+O_\delta\left(\frac{T^{\alpha-\delta}}{\vol(G/\Gamma)}+\vre^{-d^2/(a+d)}\,T^{\alpha(1-\kappa a/(a+d))}\right)
\end{align*}
for every $\eta>0$,
where
%$\vre_0(y) =\sup\{\vre>0:\, \mathcal{O}_{\vre} y\to \mathcal{O}^_{\vre} y\Gamma\hbox{ is injective}\}$,
$d =\sum_{v\in S_\infty} \dim(G_v)$ and $a$ is the H\"older exponent of the family $\{B_{e^t}\}$.
%and $\kappa =(2n_e)^{-1}$ with $n_e$ the least even integer greater than $p/2$. 
%$\alpha\in\mathbb{Q}^+$, and $P(T)$ is a nonzero polynomial,  and $\delta<\min\{\alpha,1\}$.
\end{theorem} 

%\todo{Is one factor unisotropic enough ? We usually assume no compact factors, so all of them are unisotropic}. 

\begin{remark}\label{rem:un}
Let us note the following regarding Theorem \ref{th:S-arith}. 
\begin{enumerate}
\item 
The sets $B_{e^t}$ are always H\"older well-rounded for some $a>0$
and, in fact, H\"older admissible.
As is explained in the proof of Theorem \ref{th:S-arith}, this follows from \cite[Theorem 7.19]{GN} (see also \cite{BO}).
Moreover, if either $S_f=\emptyset$, or $S_\infty=\emptyset$,  or for $v\in S_\infty$,
the representation $\rho_v$ is self-adjoint (i.e., ${}^t\rho_v(G_v)=\rho_v(G_v))$,
we can take the H\"older exponent $a=1$ (see \cite[Theorem~3.15]{GN}).
\item 
The assumption that the representation $\pi_{G/\Gamma}^0$ in $L_0^2(G/\Gamma)$ is $L^{p+}$ for some  $p>0$ holds, in the set-up of $S$-algebraic groups, in most cases. See Remark \ref{r:lp} below for further discussion.

\item If ${\sf G}_v$'s are not simply connected, we consider the simply connected covers
$\pi:\tilde G\to G$. It is known that $\pi(\tilde G)$ is normal, and $G/\pi(\tilde G)$
is Abelian of finite exponent (see \cite{bt}). 
If $\Gamma$ is finitely generated, 
$\Gamma\cap \pi(\tilde G)$ has finite index in $\Gamma$.
Applying Theorem \ref{th:S-arith} to the lattice $\tilde\Gamma=\pi^{-1}(\Gamma\cap \pi(\tilde G))$
in $\tilde G$, one can deduce the asymptotics and the error term for $\Gamma$.

\item If the height $H$ is bi-invariant under a maximal compact subgroups $K_v$ of $G_v$ when $v\in V_\infty$,  we can improve the error estimate by taking $d =\sum_{v\in S_\infty} \dim(K_v\backslash G_v)$ (see also Remark \ref{rem:main}(2)).

\item 
A further improvement in the error term in Theorem \ref{th:S-arith} can be obtained 
if in addition the local heights are each bi-invariant under a special maximal
compact subgroup (so that the Cartan decomposition holds for $G$).
In this case, we can replace $2n_e(p)$ in the error estimate by $p$, provided the $L^{p^+}$-spectrum 
in uniformly bounded, in the sense defined in \cite[\S 8.1]{GN}. This is indeed often the case.  
\end{enumerate}
\end{remark}

Taking parts (4) and (5) of the last remark into account, we note that Theorem  \ref{th:S-arith} matches the best error estimate established for bi-$K$-invariant sets in simple higher rank Lie group, but holds in much greater generality. To elucidate this point, we recall the following : 

\begin{example}\label{ex:sl}
The most basic example of the non-Euclidean lattice point counting problem is that of integral unimodular matrices, and balls w.r.t. the Hilbert-Schmidt norm, which is the Archimedean valuation. In this case, taking parts (4) and (5) of the foregoing remark 
into account, we have $ d=\dim \hbox{SL}_m(\mathbb{R})/\hbox{SO}_m(\mathbb{R})=m(m+1)/2-1$,
the balls are Lipschitz so $a=1$, and the integrability parameter is $p^+=2(m-1)$ 
namely the representation in $L^2_0(\hbox{SL}_m(\mathbb{R})/\hbox{SL}_m(\mathbb{Z}))$ is
$L^{2(m-1)+}$ (see \cite{DRS}).  
Theorem \ref{th:S-arith} then implies : 

\begin{align*}
|\hbox{SL}_m(\mathbb{Z})\cap B_T|=&\frac{\vol(B_T)}{\vol(\hbox{SL}_m(\mathbb{R})/\hbox{SL}_m(\mathbb{Z}))}
 +O_\eta\left(\vol(B_T)^{1-1/(m^3-m)+\eta}\right), \quad \eta>0.
\end{align*}
The latter estimate coincides with the best current error term obtained in \cite{DRS} for this case. 
\end{example}

%Let us begin with the following definition. 
%\begin{definition}\label{alg-gps}
%{\bf Semisimple $S$-algebraic groups.}
%Let $F$ be a locally compact non-discrete field, 
%and let $G$ be the group  
%of $F$-points of a connected semisimple linear algebraic group defined 
%over $F$, 
%with positive $F$-rank (namely containing an $F$-split torus of 
%positive dimension over $F$). We assume in addition 
%that $G$ is algebraically connected, and does not have non-trivial 
%anisotropic (i.e. compact) algebraic 
%factor groups defined over $F$. We will also assume, for simplicity, that $G^+$ is of finite index in $G$ 
%(see Remark \ref{G+}). 
%By a semisimple  $S$-algebraic group we mean any finite 
%product of the groups described above. 
%\end{definition} 

Another natural family of balls on an $S$-algebraic group $G=\prod_{v\in S} G_v$ is defined
with respect to standard $CAT(0)$ metrics on the corresponding symmetric spaces and buildings.
Let $X_v$ denote the symmetric space of $G_v$ if is Archimedean and the Bruhat--Tits
building of $G_v$ otherwise. For fixed $x=(x_v)\in \prod_{v\in S} X_v$, we set
$$
d(g)=\left(\sum_{v\in S} d_v(g x_v,x_v)^2\right)^{1/2}
$$
where $d_v$ are the standard metrics on $X_v$.
Let
\begin{equation}\label{eq:bb}
B_t=\{g\in G:\, d(g)\le t\}.
\end{equation}
Our method allows to deal with lattice subgroups of $G$ which are not necessarily irreducible.
This is related to fact that the sets $B_t$ are well-balanced (see \cite[Definition 3.17]{GN}).
Namely, the volume of the sets $B_t$ does not concentrate along 
proper direct factors of $G$ (see \cite[Theorem 3.18]{GN}). The error term in the lattice counting
problem can be estimated in terms the relative volume growth
$$
r=\max_{L<G}\limsup_{t\to\infty} \frac{\log m_L(B_t\cap L)}{\log m_G(B_t)}
$$
where the maximum is taken over proper direct factors $L$ of $G$.
If all the factors $G_v$ are not compact, then $r<1$ by \cite[Theorem 3.18]{GN}.

\begin{theorem}\label{th:S-arith2}{\bf Counting lattice points in $S$-algebraic groups: metric balls.}
Let $G$ be as in Theorem \ref{th:S-arith}, with all factors $G_v$ non-compact.
For every $\vre_0>0$ there exists $t_0>0$ such that for every lattice $\Gamma$ in $G$
for which the representation of $G_v$, $v\in S$, on the orthogonal complement of 
$L^2(G/\Gamma)^{G_v}$ is $L^{p+}$, $x,y\in G$
such that $\vre_0(x,\Gamma),\vre_0(y,\Gamma)\ge \vre_0$, and $t\ge t_0$,
\begin{align*}
|x\Gamma y^{-1}\cap
B_t|&=\frac{\vol(B_t)}{\vol(G/\Gamma)}+O_\eta\left(\vre_0^{-d^2/(1+d)}\,\vol(B_t)^{1-
(1-\sqrt{3 r^2+1}/2)/n_e(p)(1+d)+\eta}\right)
%&=\frac{T^\alpha P(\log T)}{\vol(G/\Gamma)}+O_\delta\left(\frac{T^{\alpha-\delta}}{\vol(G/\Gamma)}+\vre^{-d^2/(a+d)}\,T^{\alpha(1-\kappa a/(a+d))}\right)
\end{align*}
for every $\eta>0$,
where
%$\vre_0(y) =\sup\{\vre>0:\, \mathcal{O}_{\vre} y\to \mathcal{O}^_{\vre} y\Gamma\hbox{ is injective}\}$,
$d =\sum_{v\in S_\infty} \dim(X_v)$.
%and $\kappa =(2n_e)^{-1}$ with $n_e$ the least even integer greater than $p/2$. 
%$\alpha\in\mathbb{Q}^+$, and $P(T)$ is a nonzero polynomial,  and $\delta<\min\{\alpha,1\}$.
\end{theorem}

According to our recipe, to prove Theorems \ref{th:S-arith} and \ref{th:S-arith2} we need to establish a
decay estimate for the operator norms of the averages $\beta_t$
supported on the sets $B_t$, and establish the Lipschitz well-roundedness of the balls. 

Turning to the first ingredient, we now show that the stable quantitative mean ergodic theorem for $G$ holds in great generality.

%First recall the following result, which is consequence of the spectral transfer principle \cite{N1}  
%together with the Kunze-Stein phenomenon \cite{Co1}\cite{V}. 
%{\bf Shouldn't $G$ be a Chevalley group because Veca uses this condition for Kunze-Stein ??}

\begin{theorem}[\cite{N1}, see also \cite{GN}]\label{semisimple mean}{\bf Stable mean ergodic theorem for $S$-algebraic groups.}
Assume that the groups ${\sf G}_v$ are simply connected,
and at least one of ${\sf G}_v$'s is isotropic over $F_v$.
Consider an action of $G$ on a standard Borel probability space $(X,\mu)$
and assume that the corresponding representation $\pi_X^0$ of $G$ on the orthogonal
complement of $L^2(X)^G$ is $L^{p+}$.
Let $\beta$ be an absolutely continuous probability measure on $G$ such that
$\|\beta\|_q<\infty$ for some $q\in [1,2)$.
Then 
$$
\norm{\pi_X^0(\beta)}\le C_q \|\beta\|_q^{1/n_e(p)}
$$
where $n_e(p)$ is defined in (\ref{eq:kappa}).

When $\beta_t$ are the uniform averages supported on the sets $B_t$, we have
$$
\norm{\pi_X^0(\beta_t)}\le C^\prime_\eta m_G(B_t)^{-(2n_e(p))^{-1}+\eta},\quad \eta>0\,.
$$
In particular, if the family $B_t$ is H\"older well-rounded and
the action of $G$ on $(X,\mu)$ is ergodic, the stable quantitative mean ergodic theorem holds
with the above rate.
\end{theorem}

\begin{proof}
We recall the spectral transfer principle from \cite{N1}. 
%We set  $n_e=n_e(p)$ to be  the least even integer $\ge p/2$.
By Jensen's inequality, for real-valued functions $f_1,f_2\in \left(L^2(X)^G\right)^\perp$,
\begin{align*}
\left<\pi_X^0(\beta)f_1,f_2\right>^{n_e}&=\left(\int_G\left<\pi_X^0(g)f_1,f_2\right>d\beta(g)\right)^{n_e}
\le\int_G\left<\pi_X^0(g)f_1,f_2\right>^{n_e}d\beta(g)\\
&=\int_G\left<(\pi_X^0)^{\otimes n_e}(g)f_1^{\otimes n_e},f_2^{\otimes n_e}\right>d\beta(g)
=\left<(\pi_X^0)^{\otimes n_e}(\beta)f_1^{\otimes n_e},f_2^{\otimes n_e}\right>\\
&\le \|(\pi_X^0)^{\otimes n_e}(\beta)\| \,\|f_1\|^{n_e}\|f_2\|^{n_e}.
\end{align*}
This implies that
$$
\|\pi_X^0(\beta)\|\le \|(\pi_X^0)^{\otimes n_e}(\beta)\|^{1/n_e}.
$$
It is easy to see that $(\pi_X^0)^{\otimes n_e}$ is an $L^{2+}$-representation.
Hence, it follows from \cite{CHH} that $(\pi_X^0)^{\otimes n_e}$ is weakly contained in the regular
representation $\lambda_G$ of $G$, and
$$
\|(\pi_X^0)^{\otimes n_e}(\beta)\|\le \|\lambda_G(\beta)\|.
$$
We can now estimate $\|\lambda_G(\beta)\|$ using the Kunze-Stein inequality, namely 
$$
\|\beta*f\|_2\le C''_q \|\beta\|_q\,\|f\|_2,\quad q\in [1,2),
$$
for every $f\in L^2(G)$.
This inequality was proved for Archimedean semisimple groups
by Cowling \cite{Co1} and for non-Archimedean semisimple simply connected groups by Veca \cite{V}.
Clearly, a product of Kunze-Stein groups is a Kunze-Stein group. Indeed, if $\beta$ is a product function, the estimate for its norm as a convolution operator follows immediately by considering product functions $f$. Any $\beta$ is the $L^q(G)$-norm limit of a sequence of product functions, and the estimate follows.  
Hence every semisimple simply connected $S$-algebraic
group satisfies the Kunze-Stein inequality. This implies that 
$$
\|\lambda_G(\beta)\|\le C''_\eta m_G(B)^{-1/2+\eta},\quad \eta>0,
$$
and the and the desired norm estimate follows.

The second claim in Theorem \ref{semisimple mean}, namely the stable mean ergodic theorem for 
 families $B_t$ is an immediate consequence of the previous inequality, since the implied constant is uniform 
 for all sets $B^{\pm }_t(\vare)$.   
\end{proof}

%Now recall the  following definition 
%\begin{definition}\label{coarse}{\bf Coarse admissibility \cite{GN}.}
%A family of sets $B_t$ is called coarsely admissible if  for $t\ge t_0$,  
%for every compact set $Q\subset G$ there exists $C_Q$ such that 
%$m_G(QB_t Q) \le C_Q m_G(B_t)$. 
%\end{definition}

%Clearly, if  $B_t$ has exponential volume growth, namely 
%$\liminf \frac1t \log m_G(B_t)=\alpha > 0$, and $B_t$ are well-rounded, then $B_t^+(\vre)$ and 
%$B_t^-(\vre)$ have the same rate of exponential volume growth. Coarsely admissible families on $S$-algebraic groups are easily seen to have exponential volume growth (see \cite{GN}) and hence we obtain the following.

%In particular, we obtain 

%\begin{corollary}\label{c:mean}
%Let $G$ act on a standard Borel space $(X,\mu)$.
%Assume that the corresponding representation on $L^2_0(X)$ is $L^p$.
%Then any well-rounded family $B_t$ satisfies the stable quantitative mean ergodic theorem 
%with the rate $E(t)=O_\eta\left( m_G(B_t)^{-\kappa_p+\eta}\right)$, $\eta>0$.
%\end{corollary}

\begin{remark}\label{r:lp}
Assume that the component groups ${\sf G}_v$ are simply connected. Then  
Theorem \ref{semisimple mean} 
applies, in particular, to the actions listed below. 
%It follows that Theorem \ref{th:S-arith} also holds, when 
%the action is on a homogeneous space. 
\begin{enumerate}
\item[(i)] $G$ is a Kazhdan group, and $(X,\mu)$ is any ergodic $G$-space.
\item[(ii)] $X=G/\Gamma$ where $G$ is an almost simple connected Lie group, and $\Gamma$ is any lattice.
\item[(iii)] $X=G/\Gamma$ where $\Gamma$ is an irreducible congruence subgroup in an $S$-arithmetic lattice of a semisimple $S$-algebraic group (see Section \ref{sec:conguence} for notation).
Moreover, the parameter $p^+=p^+(\Gamma)$ is then bounded above uniformly over all congruence subgroups (namely property $\tau$ holds).
\item[(iv)] $X=G/\Gamma$ where $G$ is a connected semisimple Lie group 
all of whose factors are locally isomorphic to $\hbox{SL}_2(\RR)$, and $\Gamma$ is any irreducible lattice.
\item[(v)] $X=G/\Gamma$ with $G$ as in (iii) and $\Gamma$ 
any lattice commensurable with an irreducible congruence lattice.

\end{enumerate}
Verification of these claims depends on the fact that  matrix 
coefficients of nontrivial irreducible representations $\pi$ of almost simple simply connected
groups are in $L^p$ for some $p=p(\pi)$ (see \cite{BW,Co,H,HM,Li,LZ,Oh}). 
This implies that $L^2_0(X)$ is an $L^p$-representation for some $p>0$ provided that it has strong
spectral gap, i.e., no noncompact simple factor of $G$ has almost invariant vectors.
Hence, it remains to check that in (ii)--(v), one has the strong spectral gap.
Now (ii) follows from the work of Borel and Garland \cite{BG},
(iii) follows from the work of  Clozel \cite{CL}, (iv) was recently proved by Kelmer and Sarnak \cite{KS}, and (v) follows 
from (iii) and \cite[Lemma~3.1]{KM}. 
\end{remark}

We now complete the proof of Theorems \ref{th:S-arith} and \ref{th:S-arith2} following the recipe of \S3.1.

\begin{proof}[Proof of Theorem \ref{th:S-arith}]
For $v\in S_f$, the local heights are bi-invariant 
under a compact open subgroups $\mathcal{O}^v$ of $G_v$.
For $v\in S_\infty$, we set
$$
\mathcal{O}^v_\vre =\{g\in G_v:\, H_v(\rho_v(g_v^{\pm 1})-id)<\vre\}.
$$
Then the family of symmetric neighbourhoods
$$
\mathcal{O}_\vre=\prod_{v\in S_\infty} \mathcal{O}^v_\vre\times  \prod_{v\in S_f} \mathcal{O}^v
$$
has local dimension $d$, which equals the real dimension of $G_\infty$. Using that for every $x_1,x_2\in\hbox{M}_{n_v}(K_v)$,
$$
H_v(x_1x_2)\le H_v(x_1)H_v(x_2),
$$
we deduce that for $g,h\in \mathcal{O}_\vre$ and $b\in B_T$,
$$
H(gbh)\le H(g)H(b)H(h)\le (1+\vre)^{2|S_\infty|}H(b)\le (1+\vre)^{2|S_\infty|}T.
$$
Hence,
\begin{equation}\label{eq:hh} 
B_T^+(\vre)=\mathcal{O}_\vre B_T \mathcal{O}_\vre \subset B_{(1+\vre)^{2|S_\infty|}T}.
\end{equation}
Similarly,
\begin{equation}\label{eq:hh2} 
B_T^-(\vre)\subset B_{(1+\vre)^{-2|S_\infty|}T}.
\end{equation}
Now, if $H_v$ are constant on $G_v$ for $v\in S_\infty$, 
then it is clear that 
the family $\{B_{e^t}\}$ is even Lipschitz admissible.
  %\chck{ question : doesn't constancy of $H_v$ mean that all the Archimedean places are compact ?????%yes, the volume is a step function in this case, so we need a separate argument}
Otherwise, the function $t\mapsto \log \vol(B_{e^t})$ is uniformly H\"older
(see \cite[Theorem 7.19]{GN} or \cite{BO}),
and it follows from  (\ref{eq:hh}) and (\ref{eq:hh2}) that the family $\{B_{e^t}\}$
is H\"older admissible (and, in particular, H\"older well-rounded).
Hence, combining Corollary \ref{cor:shift} and Theorem \ref{semisimple mean},
the result follows.
\end{proof}

\begin{proof}[Proof of Theorem \ref{th:S-arith2}]
Let $\mathcal{O}_\vre=\{g\in G:\, d(g)<\vre\}$. It follows from the triangle inequality that
$$
B_t^+(\vre)\subset B_{t+\vre}\quad\hbox{and}\quad B_t^-(\vre)\supset B_{t-\vre}.
$$
If $S_\infty\ne \emptyset$, then the function $t\mapsto \log
m_G(B_t)$ is uniformly Lipschitz by \cite[Theorem 3.18]{GN}. Hence, the family $B_t$ is
Lipschitz well-rounded in this case, and in fact Lipschitz admissible. Otherwise, the family $B_t$ is bi-invariant under 
a compact open subgroup of $G$ and, in particular, Lipschitz well-rounded as well.

In view of Corollary \ref{cor:shift} (and Remark \ref{rem:main}(2)), it remains to prove the quantitative mean ergodic theorem
for the uniform averages $\beta_t$ along the sets $B_t$.
Namely, we need to show that for every $f\in L^2_0(X)$,
\begin{equation}\label{eq:b}
\|\pi^0_{G/\Gamma}(\beta_t)f\|_2\ll_\eta m_G(B_t)^{(1-\sqrt{3 r^2+1}/2)/n_e(p)+\eta}\|f\|_2, \quad \eta>0.
\end{equation} 
 
For $J\subset S$, we set $G_J=\prod_{j\in J} G_j$ and $G^J=\prod_{j\notin J} G_j$.
We observe that $L^2(X)=\sum_{J\subset I} \mathcal{H}_J$ where $\mathcal{H}_J$ are
orthogonal closed $G$-invariant subspaces of $L^2(X)$ such that every vector in 
$\mathcal{H}_J$ is fixed by $G_J$ and there are no nonzero vectors fixed by $G_j$, $j\notin J$.
We note that the representation of $G^J$ on $\mathcal{H}_J$ is $L^{p+}$.
Indeed, the representation of each $G_j$, $j\notin J$, on $\mathcal{H}_J$ is $L^{p+}$,
and every irreducible representation of $G^J$ appearing in the decomposition of $\mathcal{H}_J$
is a tensor product of irreducible representations of the factors.

For $f\in \mathcal{H}_J$ and $J\neq S$,
\begin{align*}
\pi^0_{G/\Gamma}(\beta_t)f(x)=\frac{1}{m_G(B_t)}\int_{B^J_t}
m_{G_J}(B_{J,\sqrt{t^2-d(g)^2}})f(g^{-1}x)\,dm_{G^J}(g)
=(\beta_t* f)(x)
\end{align*}
where $\beta_t(g)= \frac{1}{m_G(B_t)} m_{G_J}(B_{J,\sqrt{t^2-d(g)^2}})\chi_{B^J_t}$.
By Theorem \ref{semisimple mean}, for every $q\in [1,2)$,
\begin{equation}\label{eq:m}
\|\pi^0_{G/\Gamma}(\beta_t)f\|_2 \ll_q \|\beta_t\|_q^{1/n_e(p)}\|f\|_2.
\end{equation}
We have 
\begin{equation}\label{eq:m2}
\|\beta_t\|_q=\frac{1}{m_G(B_t)} \left(\int_{B^J_t} m_{G_J}(B_{J,\sqrt{t^2-d(g)^2}})^q\, dm_{G^J}(g)  \right)^{1/q}.
\end{equation}
Let 
\begin{align*}
v_J=\lim_{t\to\infty} \frac{1}{t}\log m_{G_J}(B_{J,t})\quad\hbox{and}\quad
v^J=\lim_{t\to\infty} \frac{1}{t}\log m_{G^J}(B^J_{t}). 
\end{align*}
This limits exist by \cite[Lemma 7.11]{GN}. We have 
\begin{equation}\label{eq:est}
m_G(B_{J,t}) \ll_\eta e^{(v_J+\eta)t}\quad\hbox{and}\quad m_G(B^J_{t}) \ll_\eta e^{(v^J+\eta)t}
\end{equation}
for all $\eta>0$ and $t\ge 0$, and
\begin{equation}\label{eq:est0}
m_G(B_{J,t}) \gg_\eta e^{(v_J-\eta)t}\quad\hbox{and}\quad m_G(B^J_{t}) \gg_\eta e^{(v^J-\eta)t}
\end{equation}
for all $\eta>0$ and $t\ge t(\eta)$.

We claim that for $\eta>0$ and $t\ge 0$,
\begin{align}\label{eq:m1}
\int_{B^J_t} m_{G_J}(B_{J,\sqrt{t^2-d(g)^2}})^q\, dm_{G^J}(g)
\ll_\eta \exp\left(\left(\sqrt{(q v_J)^2+(v^J)^2}+\eta\right)t\right),
\end{align}
and for $\eta>0$ and $t\ge t(\eta)$,
\begin{align}\label{eq:m2}
\int_{B^J_t} m_{G_J}(B_{J,\sqrt{t^2-d(g)^2}})^q\, dm_{G^J}(g)
\gg_\eta \exp\left(\left(\sqrt{(q v_J)^2+(v^J)^2}-\eta\right)t\right).
\end{align}
To verify these claims, we need to consider two cases: when $G^J$ has at least one
Archimedean factor, and when $G^J$ consists of non-Archimedean factors.

In the first case, we observe that the sets $B^J_t$ are admissible by \cite[Theorem 3.18]{GN}
and by \cite[Proposition 3.13]{GN}, we have   $m_{G^J}=\int_0^\infty m^J_t\, dt$ where 
$m^J_t$ be a measure supported on $S_t^J=\{g\in G^J: d(g)=t\}$. Moreover, it follows
from the admissibility that
\begin{equation}\label{eq:est2}
m_t^J(S^J_t) \ll_\eta e^{(v^J+\eta)t}
\end{equation}
for all $\eta>0$ and $t\ge 0$. Then by (\ref{eq:est}) and (\ref{eq:est2}),
\begin{align*}
\int_{B^J_t} m_{G_J}(B_{J,\sqrt{t^2-d(g)^2}})^q\, dm_{G^J}(g) &=\int_0^1 t\, m_{G_J}(B_{J,t\sqrt{1-u^2}})^q
m_{tu}^J(S^J_{tu})\, du\\
&\ll_\eta \int_0^1 t\exp((qv_J\sqrt{1-u^2}+v^Ju+\eta)t)\,du\\ 
&\ll_\eta \exp\left(\left(\sqrt{(q v_J)^2+(v^J)^2}+\eta\right)t\right)
\end{align*}
for every $\eta>0$, where the last estimate is obtained by maximising 
the function $\phi(u)= qv_J\sqrt{1-u^2}+v^Ju$.
This proves (\ref{eq:m1}). To prove the opposite inequality, we note that
\begin{equation*}\label{eq:est3}
m_t^J(S^J_t) \gg_\eta e^{(v^J-\eta)t}
\end{equation*}
for all $\eta>0$ and $t\ge t(\eta)$ (see \cite[Proof of Theorem 3.18]{GN}).
Taking a sufficiently small neighbourhood $U$
of the point $u_0\in (0,1)$ of maximum of the function $\phi$, we obtain
\begin{align*}
\int_{B^J_t} m_{G_J}(B_{J,\sqrt{t^2-d(g)^2}})^q\, dm_{G^J}(g) &\ge \int_U t\, m_{G_J}(B_{J,t\sqrt{1-u^2}})^q
m_{tu}^J(S^J_{tu})\, du\\
&\gg_\eta \exp\left(\left(\sqrt{(q v_J)^2+(v^J)^2}-\eta\right)t\right)
\end{align*}
for every $\eta>0$, which proves (\ref{eq:m2}).

In the case when $G^J$ is a product of non-Archimedean factors, we have
$$
\int_{B^J_t} m_{G_J}(B_{J,\sqrt{t^2-d(g)^2}})^q\, dm_{G^J}(g) =
\sum_{u \in [0,1]:\, m_{G^J}(S^J_{tu})\ne 0} m_{G_J}(B_{J,t\sqrt{1-u^2}})^q
m_{G^J}(S^J_{tu}).
$$
Since $|d(G^J)\cap [0,t]|\ll t^d$ for some $d>0$ and $t\ge 1$, the inequality (\ref{eq:m1})
follows from (\ref{eq:est}). Since the gaps between distances $d(G^J)$ are uniformly bounded,
there exists $u_t\in (0,1)$ such that $S_{tu_t}^J\ne \emptyset$ and $|u_t-u_0|=O(1/t)$.
Since $u_0\in (0,1)$, we have  $|\phi(u_t)-\phi(u_0)|=O(1/t)$.
As in \cite[Lemma 7.11]{GN}, when $S^J_t\ne \emptyset$, we have
\begin{equation*}\label{eq:est3}
m^J(S^J_t) \gg_\eta e^{(v^J-\eta)t}
\end{equation*}
for all $\eta>0$ and $t\ge t(\eta)$. Then
\begin{align*}
\int_{B^J_t} m_{G_J}(B_{J,\sqrt{t^2-d(g)^2}})^q\, dm_{G^J}(g) &\gg_\eta
\exp\left(\left(qv_J\sqrt{1-u_t^2}+v^Ju_t-\eta\right)t\right)\\
&\ge \exp\left(\left(qv_J\sqrt{1-u_0^2}+v^Ju_0-O(1/t)-\eta\right)t\right)
\end{align*}
for every $\eta>0$. This implies (\ref{eq:m2}).

Since
$$
m_G(B_t)=\int_{B^J_t} m_{G_J}(B_{J,\sqrt{t^2-d(g)^2}})\, dm_{G^J}(g),
$$
the estimates (\ref{eq:m1}) and  (\ref{eq:m2}) imply that
$$
v:=\lim_{t\to\infty} \frac{1}{t}\log m_{G}(B_{t})=\sqrt{(v_J)^2+(v^J)^2}.
$$
Setting $r_J=v_J/v$, we obtain 
\begin{align*}
\left(\int_{B^J_t} m_{G_J}(B_{J,\sqrt{t^2-d(g)^2}})^q\, dm_{G^J}(g)  \right)^{1/q}
&\ll_\eta \exp\left(\left(\sqrt{(v_J)^2+(v^J/q)^2}+\eta\right)t\right)\\
&\ll_\eta m_G(B_t)^{\sqrt{r_J^2+q^{-2}(1-r_J^2)}+\eta}
\end{align*}
for every $\eta>0$.

Finally, it follows from (\ref{eq:m}) and (\ref{eq:m1}) that
$$
\|\pi^0_{G/\Gamma}(\beta_t)f\|_2\ll_\eta
m_G(B_t)^{\left( 1-\sqrt{r_J^2+(1-r_J^2)/q^2}\right)/n_e(p)+\eta}\|f\|_2,\quad
\eta>0.
$$
Since this estimate holds for every $q\in [1,2)$, the claim (\ref{eq:b}) follows.
This completes the proof of the theorem.
\end{proof}
 
\section{Congruence subgroups and density hypothesis}\label{sec:conguence}
Let ${\sf G}\subset\hbox{GL}_m$ be a connected semisimple  algebraic group defined over a number field $F$.
We fix a finite set $S$ of places of $F$ which contains
all Archimedean places $V_\infty$ and the group $G=\prod_{v\in S} {\sf G}(F_v)$ is noncompact.
Then $\Gamma={\sf G}(\mathcal{O}_S)$, where $\mathcal{O}_S$ is the ring of $S$-integers,
is a lattice in $G$. Given an ideal $\mathfrak{a}$ of $\mathcal{O}_S$,
we introduce a congruence subgroup
$$
\Gamma(\mathfrak{a})=\{\gamma\in\Gamma:\, \gamma=I \mod\mathfrak{a}\}.
$$
The height function $H$ on $G$ is defined as in (\ref{eq:H}) and $B_T=\{g\in G:\, H(g)<T\}$.
If $\sf G$ is simply connected and $F$-simple, then
 property $(\tau)$, established in full generality by Clozel \cite{CL}, shows that there exists
$p>0$ such that all representations in  $L_0^2(G/\Gamma(\mathfrak{a}))$ as $\mathfrak{a}$ varies are $L^{p+}$-representations,  with $p^+$ independent of $\mathfrak{a}$.
Hence, Theorem \ref{th:S-arith} immediately implies the following uniformity result in counting lattice points, generalising \cite{NS}.

\begin{theorem}\label{uniform}{\bf Uniformity in counting over congruence groups.}
Keeping the notation of the previous paragraph, if $\sf G$ is simply connected
and $F$-simple, there exists $T_0>0$ such that for every $\gamma_0\in\Gamma$, all ideals
$\mathfrak{a}$ of $\mathcal{O}_S$, and $T\ge T_0$,  
\begin{align*}
\abs{\set{\gamma\in \gamma_0\Gamma(\mathfrak{a})\,:\, H(\gamma)< T}}
&=\frac{\vol(B_T)}{[\Gamma:\Gamma(\mathfrak{a})]}+O_\eta\left(\vol(B_T)^{1-(2n_e(p))^{-1}
    a/(a+d)+\eta}\right)
%&=\frac{T^{\alpha}P(\log
%  T)}{[\Gamma:\Gamma(\mathfrak{a})]}\left(1+O\left( T^{-\delta}\right)\right)\\
\end{align*}
for every $\eta>0$. Here the measure on $G$ is normalised so that $\vol(G/\Gamma)=1$,
$a$ is the H\"older exponent for the family  $\{B_{e^t}\}$, $d=\sum_{v\in V_\infty} \dim {\sf G}(F_v)$, and the implied constant is independent of the ideal $\mathfrak{a}$. 
\end{theorem}

Let us now recall the following conjecture:

\begin{conjecture}[\cite{SX},\cite{S}]\label{main} 
 For any semisimple algebraic $\QQ$-group 
${\sf G}\subset \hbox{\rm GL}_m$,  
the following upper bound holds, uniformly in $N\in \NN$ (for any fixed choice of norm)
$$\abs{\set{
\gamma\in \Gamma(N)\,\,:\,\, \norm{\gamma}< T}}= O_\eta\left(
\frac{T^{\alpha+\eta}}{[\Gamma(1):\Gamma(d)]}+T^{\alpha/2}\right),\quad \eta>0,
$$
where $\Gamma(N)$ are the principal congruence group mod $N$ in 
${\sf G}(\ZZ)$, and 
$$\alpha=\limsup_{T\to \infty}
\frac{\log \vol(B_T)}{\log T}.
$$
\end{conjecture}

From Theorem \ref{uniform}, we obtain the following result  :
 
\begin{corollary}\label{cor}
 For any semisimple  $\QQ$-simple algebraic group ${\sf G}\subset \hbox{\rm GL}_m$,  and any fixed choice of norm, 
$$\abs{\set{\gamma\in \Gamma(N)\,:\,  \norm{\gamma}< T}}=O_\eta \left(
\frac{T^{\alpha+\eta}}{[\Gamma(1) : \Gamma(N)]}+
T^{(\alpha+\eta)(1 -\theta)}\right),\quad \eta>0,$$
where $\theta=(2n_e(p))^{-1}/\left(1+\dim \left({\sf G}(\mathbb{R})/K\right)\right)$ and 
$K$ is a maximal compact subgroup of ${\sf G}(\mathbb{R})$. Here $N\in \NN$ is arbitrary, and the implied constant is independent of $N$.  $\alpha$ is the rate of volume growth of the norm balls. 
\end{corollary}

%{\bf Note : it is also possible to replace $\kappa$ in the bound the error explicitly in terms of the split rank only or the bound towards the Ramanujan conjecture, as in \cite{NS}. But then the strong spectral gap and balancedness issue must enter explicitly -  this is a problem already in the real case .}

\begin{proof}
To deduce this corollary from Theorem \ref{uniform}, we consider the simply
connected cover $\pi:\tilde{\sf G}\to {\sf G}$ and note that $\pi(\tilde{\sf G}(\mathbb{Z}))$
is commensurable to ${\sf G}(\mathbb{Z})$. Hence, without loss of generality, we may
assume that ${\sf G}$ is simply connected.

With respect to a suitable basis of $\mathbb{R}^d$, ${\sf G}$ is self-adjoint
and there exists a maximal compact subgroup $K\subset {\sf G}(\mathbb{R})$
such that $K\subset \hbox{SO}_m(\mathbb{R})$.
We note that the estimate in the theorem is independent of a choice of the norm.
Hence, we may assume that $\|\cdot \|$ is a Euclidean norm with respect to
the above basis. Then the sets $\{g\in {\sf G}(\mathbb{R}):\, \|g\|\le e^t\}$
are Lipschitz admissible (see \cite[Theorem 3.15]{GN}), and bi-$K$-invariant.
Hence, the corollary follows from Theorem \ref{th:S-arith} and Remarks \ref{rem:un}(1),(4).
\end{proof}

%Corollary \ref{cor} constitutes a generalisation of the following example, which was considered in \cite{NS}.
%\begin{example}
%For ${\sf G}=\hbox{SL}_m$,
%$$
%|\{\gamma\in\Gamma(N):\,
%\|\gamma\|\le T\}|=O_\eta\left(\frac{T^{(m^2-m)+\eta}}{|\Gamma(1):\Gamma(N)|}+ T^{(m^2-m)+\eta-1/(m+1)} \right)
%$$
%for every $\eta>0$ (cf. Example \ref{ex:sl}). Here the implied constant is independent of $N$. 
%\end{example}

Comparing Corollary \ref{cor} and Conjecture \ref{main}, let us note the following.

\begin{enumerate}
\item The second term appearing in the estimate stated in  Conjecture \ref{main} is the best possible, and is asserted only for the {\it principal} congruence groups. 
It may fail if more general finite-index subgroups are admitted 
as demonstrated by the construction of exceptional eigenvalues in 
\cite{BS}, \cite{BLS}. Thus the conjecture predicts 
a regularity property of the lattice point counting problem 
satisfied specifically by principal congruence subgroups. 
%\item Corollary \ref{cor} does not assume that the lattice is 
%arithmetic, 
%or that the finite index subgroups are 
%congruence subgroups, principal or otherwise - it only depends on
% the validity of property $\tau$ for the family. 
%In this generality, typically the best possible 
%error term must satisfy $\beta > \alpha/2$ as noted in (1).
%by \cite{BS}\cite{BLS}. 
\item The proof of Theorem \ref{error estimate}  
above can not produce the second term called for in Conjecture \ref{main}, which is the square 
root of the volume of the ball. Indeed, Theorem \ref{error estimate} will still 
yield an error term greater than the square root of the volume even if the spectral gap 
is the largest possible, namely all the representations occurring 
in $L^2_0({\sf G}(\RR)/\Gamma(N))$ are tempered. An error term with this quality can be established only for a smooth 
weighted form of the lattice point counting problem, see \cite{NS}.
On the upside, Theorem \ref{uniform} actually gives an error 
estimate, uniform over all $\Gamma(N)$ and their cosets, namely a lower bound 
 as well as an upper bound. 

\end{enumerate}

The cases 
where Conjecture \ref{main} has been verified are the set of 
 arithmetic lattices in $\hbox{SL}_2(\RR)$ and $\hbox{SL}_2(\CC)$
 \cite[Thm. 1]{SX}. Note that in 
 those cases the conjecture was established {\it without } 
assuming a spectral gap, and indeed was used to derive it,  
thus giving an independent approach to uniform spectral gaps 
for congruence subgroups \cite[Corollary 2]{SX}. 

An important application of Conjecture \ref{main} is to the density 
hypothesis, which bounds the multiplicities of the 
${\sf G}(\RR)$-representations occurring in $L^2({\sf G}(\RR)/\Gamma(N))$ (see \cite{DG-W1}\cite{DG-W2} for a discussion of this problem). 
For an irreducible non-trivial $\pi$, we let $p_K^+(\pi)$ denote the infimum 
over $p \ge 2$ such that the $K$-finite matrix coefficients of $\pi$ 
are in $L^p({\sf G}(\RR))$. Let $m(\pi,\Gamma(N))$ denote the multiplicity 
in which $\pi$ occurs in $L^2({\sf G}(\RR)/\Gamma(N))$. 
Consider the following ``density hypothesis'': 

\begin{conjecture}[\cite{SX}, \cite{S}]\label{density}%{\bf Density hypothesis}
With notation as in Conjecture \ref{main}, assume that  ${\sf G}(\ZZ)$ is cocompact.
Then for all $\eta > 0$,
$$m(\pi,\Gamma(N))=O_\eta\left([\Gamma(1) :\Gamma(N)]^{(2/p_K^+(\pi))+\eta}\right).$$
\end{conjecture}

When $G$ has real rank one and $\Gamma$ is cocompact, the density 
hypothesis was shown to follow from Conjecture \ref{main} (see \cite{SX}). 

\begin{remark}
The method used in \cite{SX} can be combined  
with Theorem \ref{uniform} to give an alternative proof of a result 
in \cite{SX}, which states that the estimate in Conjecture \ref{density} holds 
with the power $2/p_K(\pi)$ of $[\Gamma(1):\Gamma(N)]$ replaced 
by a weaker estimate. 
However, since Theorem \ref{uniform} does not require the 
rank-one hypothesis, one expects that it 
 can be used to establish a multiplicity bound in 
terms of an appropriate power of $[\Gamma(1):\Gamma(N)]$ 
more generally, for groups of arbitrary rank.

\end{remark} 

\section{Rational points, and Kunze-Stein phenomenon on adele groups}\label{sec:rat}

Let $F$ be an algebraic number field. Keepint the notation from \S4.1, we define
the height of an $F$-rational vector $x=(x_1,\dots,x_d)\in F^d$  by
$$
H(x)=\prod_{v\in V} H_v(x),
$$
where the local heights $H_v$ are defined as in (\ref{eq:H1})--(\ref{eq:H2}).
For example, if $F=\mathbb{Q}$ and $x\in\mathbb{Q}^\times\cdot(x_1,\ldots,x_d)$
where $x_1,\ldots,x_d\in\mathbb{Z}$ and $\gcd(x_1,\ldots,x_d)=1$, then
$$
H(x)=\left(|x_1|^2+\cdots + |x_d|^2 \right)^{1/2}.
$$
%\chck{Isn't this formula Archimedean ?? what about local fields and the $p$-adic absolute value???
%this is only an example in the case when F=Q to show that the height is natural.}
The number of rational points with bounded height lying on a projective variety is finite,
and one of the fundamental problems in arithmetic geometry is to
determine its asymptotics (see, for instance, \cite{ts}).

Let ${\sf G}\subset\hbox{GL}_m$ be a semisimple algebraic group defined over $F$. Then the cardinality of the set $\{\gamma\in {\sf G}(F):\, H(\gamma)\le T\}$ is finite,
and we are interested in its asymptotic as $T\to\infty$.
The set ${\sf G}(F)$ embeds discretely in the group ${\sf G}(\mathbb{A})$
of adeles as a subgroup of finite covolume, and the height $H$ extends to ${\sf G}(\mathbb{A})$.
We set 
\begin{equation}\label{eq:B_t}
B_T=\{g\in {\sf G}(\mathbb{A}):\, H(g)\le T\}.
\end{equation}
To state our main result, we note that it follows from \cite{CL} that 
provided ${\sf G}$ is simply connected and $F$-simple,
the representation $\pi^0_{{\sf G}(\mathbb{A})/{\sf G}(F)}$ in 
$L^2_0({\sf G}(\mathbb{A})/{\sf G}(F))$ is $L^{p+}$ for some $p>0$
(we will explain this in detail in the proof of Theorem \ref{rational points} below).

\begin{theorem}\label{rational points}
Assume that the group ${\sf G}$ is simply connected and $F$-simple.
%\chck{why simple and not semisimple?? is this Clozel's assumption ??
%yes, this guarantees $L^p$-condition. in general, we have to worry about
%well-balancedness.} 
 Then
$$
|{\sf G}(F)\cap B_T|=\frac{m_{{\sf G}(\Adele)}(B_T)}{m_{{\sf G}(\Adele)}({\sf G}(\mathbb{A})/{\sf
    G}(F))}+O_\eta \left(m_{{\sf G}(\Adele)}(B_T)^{1-(2n_e(p))^{-1} a /(2d+2a)+\eta}\right)
$$
for every $\eta>0$, where $a$ is the H\"older exponent of the family $\{B_{e^t}\}$
and $d=\sum_{v\in V_\infty}\dim_\RR{\sf G}(F_v)$.
\end{theorem}

\begin{remark}
Let us note the following regarding Theorem \ref{rational points} (compare with Remark \ref{rem:un}). 
%(see Remark \ref{rem:main}(1)).
\begin{enumerate}
\item If the height $\prod_{v\in V_\infty} H_v$ is bi-invariant under 
a maximal compact subgroup $K$, then $d$ in the error estimate
can be replaced by %$d=\sum_{v\in V_\infty}\dim_\RR{\sf G}(F_v)$ 
$d-\dim_\RR K$. 
\item 
If in addition the local heights $H_v$ are each bi-invariant under a special maximal
compact subgroup $K_v$ of $G_v$, the error term in Theorem \ref{rational points} can be improved by  replacing $2n_e(p)$ by $p$, provided the $L^{p^+}$-spectrum is uniformly bounded in the sense of 
\cite[\S 8.1]{GN}.
\item If ${\sf G}\subset \hbox{GL}_m$ is self-adjoint (namely invariant under transpose)
% For example, the standard embedding of
%$\hbox{SL}_m$ is self-adjoint.
then the family $B_{e^t}$ are Lipschitz well-rounded. Indeed, 
this is clear when the height $\prod_{v\in V_\infty} H_v$ is constant on $G_\infty$. Otherwise, we have
Lipschitz estimate at the Archimedean places \cite[Proposition~7.5]{GN},  and \cite[Theorem 3.15]{GN} (or the argument 
in \cite[Proposition~2.19(2)]{GO2}) yield the desired Lipschitz estimate. Then in Theorem \ref{rational points},
one can set $a=1$. 
\end{enumerate}
\end{remark}

We note that the main term of the asymptotics of the number of rational points on semisimple group varieties was computed in \cite{STBT2} (using direct spectral expansion of the automorphic kernel) and in \cite{GMO} (using mixing). However, as noted in \S 1.3, these methods do not produce an error term of  the same quality as Theorem \ref{rational points}.

%\begin{example}
%$$
%|\{x\in \hbox{SL}_d(F):\, H(x)<T\}|=c\,T^{n^2}+O(T^{n^2-1/2}+T^{n^2-2 n\kappa/(n-1)}).
%$$
%{\bf compute what is $\kappa$?}
%{\bf add more details or skip?}
%\end{example}

% Given an algebraic 
%group ${\sf G}$ defined over $F$, we let $G_v={\sf G}(F_v)$
% denote the group of $F_v$ 
%points, and 
%${\sf G}(F)$ the group of $F$-rational points. 
%We assume that  ${\sf G}$ is embedded in a matrix group via
%an $F$-rational representation $\rho:{\sf G}\to\hbox{GL}_d$.
%This defines an integral model for ${\sf G}$, and 
%we denote by ${\sf G}(\mathcal{O}_v)$, $v\in V_f$,  the group of $\mathcal{O}_v$-points. 

\subsection{The structure of semisimple adele groups}

The adele group $G={\sf G}(\Adele)$ is defined as the 
direct product $G=G_\infty \times G_f$, where 
$G_\infty=\prod_{v\in V_\infty} G_v$, and 
$G_f$ is the restricted direct product 
${\prod}^\prime_{v\in V_f} (G_v,K_v)$
 of the locally compact 
groups ${\sf G}(F_v)$ w.r.t. the compact 
open subgroups $K_v$, which for almost all $v\in V_f$ satisfies 
$K_v={\sf G}(\mathcal{O}_v)$ (for a fixed integral model for ${\sf G}$). 
Thus each element of the restricted direct product $G_f$ can be 
identified with a sequence $(g_v)_{v\in V_f}$, such that 
$g_v\in {\sf G}(\mathcal{O}_v)$ for almost every $v\in V_f$, and 
$G$ is a locally compact $\sigma$-compact group. 
We recall that 
if we choose Haar measures $m_v$ 
on each $G_v$, normalised so that $m_v(K_v)=1$  for $v\in V_f$, 
and define 
the measure $m_{G_f}$ via 
the construction of restricted product of measure spaces, namely  
$$\left(G_f, {\prod}_{v\in V_f}K_v,m_{G_f}\right)=
{\prod}^\prime_{v\in V_f} (G_v, K_v, m_v),$$ 
then $m_{G_f}$ is a Haar measure 
on $G_f$ (see \cite{Bl} and \cite{Mo} for further details on this construction).  
Haar measure on $G=G_\infty\times G_f$ is then the direct product  $m_{G_\infty}\times m_{G_f}$.

Now assume that ${\sf G}$ is a semisimple simply connected algebraic group defined over $F$.
We choose the  family of subgroups $K_v$ so that an analogue of the Iwasawa decomposition holds
for $G$. Recall that by \cite{t},
for almost all $v$, ${\sf G}(\mathcal{O}_v)$ is hyperspecial maximal compact subgroup
of $G_v$. For every $v\in V$, we fix a maximal compact subgroup $K_v$ of $G_v$
so that $K_v$ is special for all $v\in V_f$ and $K_v={\sf G}(\mathcal{O}_v)$ for almost all $v$.
Then for every $v$, the Iwasawa decomposition $G_v=K_vP_v$ holds where $P_v$ is a closed
amenable subgroup given by $P_v={\sf Z}(K_v){\sf U}(K_v)$
where $\sf Z$ is the centraliser of a suitable maximal $F_v$-split torus in $\sf G$,
and $\sf U$ is the subgroup generated by positive root groups (see \cite{t} in the non-Archimedean case).
Setting $K=\prod_{v\in V} K_v$ and $P={\prod}_{v\in V}' (P_v, P_v\cap K_v)$, we have
the Iwasawa decomposition $G=KP$ for the adele group. 
For $g\in G$, we denote by $p(g)$ the $P$-component of $g$ with respect to the Iwasawa decomposition. 
The element $p(g)$ is well-defined modulo $P\cap K$, and the modular function of $P$ is constant on each coset of $P\cap K$.

\subsection{Harish-Chandra function} 
The Harish-Chandra function is the $K$-bi-invariant function on $G$ defined by
$$
\Xi_G(g)=\int_K \Delta_P(p(gk))^{-1/2}\,dk
$$
where $\Delta_P$ is the modular function of $P$.

$\Xi_G$ plays fundamental role in analysis on semisimple groups over the adeles. 
First, let us note that $\Xi_G(g)=\prod_{v\in V} \Xi_{G_v}(g_v)$, since $\Delta_P(p)=\prod_{v\in V}\Delta_{P_v}(p_v)$, $K=\prod_{v\in V} K_v$, and Haar probability measure on $K$ is the product of the Haar probability measures on $K_v$, $v\in V$. 
Second, note that the Cowling--Haagerup--Howe argument \cite{CHH}, which is valid 
for every group with an Iwasawa 
decomposition (see \cite[\S 5.1]{GN}), shows that 
matrix coefficients  of $K$-finite vectors of a unitary representation $\pi$ 
which is weakly contained in the 
regular representation are estimated by
\begin{equation}\label{eq:CHH}
\inn{\pi(g)\xi,\eta}\le \sqrt{(\dim \left< K\xi\right>) (\dim \left< K\eta\right>)}
\norm{\xi}\norm{\eta}\Xi_G(g)
\end{equation}
Although the Harish-Chandra function for semisimple groups over local fields is in $L^{2+\vre}(G)$
for every $\vre>0$, this is no longer the case for the group of adele points. Instead, we have the following   result. 

\begin{proposition}\label{p:xi}{\bf Integrability of the Harish Chandra function on adele groups.}
Keeping the assumption and notation of the previous subsection, $\Xi_G\in L^{4+\vre}(G)$ for every $\vre>0$, where $G={\sf G}(\Adele)$.
\end{proposition}

Let us immediately note that for general simple groups over the adeles the exponent $4$ is the best
possible. In particular such groups do not satisfy the standard Kunze-Stein inequality, which requires
that the integrability exponent be equal to $2$, but only a weaker version of it (see Theorem \ref{th:ks}
below).  

\begin{proposition}\label{optimal}{\bf Optimality of the integrability exponent.}
The exponent $4$ in Proposition \ref{p:xi} is optimal for ${\sf G}=\hbox{\rm PGL}_2$.
\end{proposition}
\begin{proof}
Let $a_v=\hbox{diag}(s_v,1)$ where $s_v$ denotes the uniformiser of $F_v$. 
We have the Cartan decomposition 
$$
G_v=K_v\, \{a_v^n\}_{n\ge 0}\,K_v
$$
and 
the decomposition for the Harish-Chandra function
$$
\Xi(g)=\prod_{v\in V} \Xi_{v}(g_v)
$$
where $\Xi_v$'s are the Harish-Chandra functions of $G_v$'s.
Using the estimates
\begin{align*}
\Xi_v(a_v^n)\ge c_1 q_v^{-n/2}\quad\hbox{and}\quad \vol(K_v a_v^n K_v)\ge c_2 q_v^n
\end{align*} 
with some $c_1,c_2>0$,
we conclude that for $p>2$,
$$
\int_{G_v} \Xi_{v}^p\, dm_v\ge 1+\sum_{n\ge 1} (c_1^p c_2)q_v^{-pn/2+n}\ge 1+c_3q_v^{1-p/2}
$$
for some $c_3>0$. Since the Dedekind zeta function $\prod_v (1-q_v^{-s})^{-1}$ has a pole at $s=1$, it follows that
the product
$\prod_{v\in V_f} \int_{G_v} \Xi_{v}^p \,dm_v$
diverges when $p\le 4$.
\end{proof}

\begin{proof}[Proof of Proposition \ref{p:xi}]
We have
$$
\Xi_{G}(g)=\prod_{v\in V} \Xi_{G_v}(g_v),\quad g=(g_v)\in G.
$$
It is well-known that the local Harish-Chandra functions $\Xi_{G_v}$ are in $L^{2+\vre}(G_v)$ for every $\vre>0$.
Since $\Xi_{G_v}$ are all bounded by $1$, it suffices to prove that for some finite $V_0\subset V$ containing the Archimedean places,
the function $\prod_{v\notin V_0} \Xi_{G_v}$ is in $L^p(G)$ for $p>4$.

In the proof, we use the explicit description of Cartan decomposition over non-Archimedean fields,
 which we now briefly recall (see \cite{t} for details).
Since $K_v$ can be assume to be special, $G_v=K_v{\sf Z}(F_v)K_v$ where
$\sf Z$ is the centraliser of a suitable $F_v$-split torus $\sf S$ in $\sf G$.
Moreover, for almost all $v$, ${\sf G}$ is split over an unramified extension of $F_v$,
so that $G_v=K_v{\sf S}(F_v)K_v$. 
We assume that this decomposition holds for all $v\notin V_0$.
Let $\Pi_v$ be the set of simple roots for ${\sf S}(F_v)$ and
$$S_v^+=\{s\in {\sf S}(F_v):\, |\chi(s)|_v\ge  1\hbox{ for $\chi\in \Pi_v$}\}.$$
Then we also have $G_v=K_v S_v^+ K_v$. 
We will use the following basic bound for the Harish-Chandra function:
\begin{equation}\label{Xi estimate}
\Xi_{G_v}(a_v)\le c_\vre\, \Delta_{P_v}(a_v)^{-1/2+\vre},\quad a_v\in S_v^+,\; \vre>0, 
\end{equation}
where $\Delta_{P_v}$ is the modular function of the group $P_v$
(see \cite[Thm. 4.2.1]{Si}). It is clear from the proof in \cite{Si}
that the constant $c_\vre>0$ can be chosen to be bounded uniformly in $v$. We have
$$
\Delta_{P_v}(a_v)=\prod_{\chi\in\Pi_v} |\chi (a_v)|_v^{ n_{\chi,v}}
$$
for some strictly positive integers $n_{\chi,v}\in\mathbb{N}$.
The Haar measure  of the double coset $K_v a_v K_v$ (subject to 
the usual normalisation $m_{G_v}(K_v)=1$ for $v\in V_f)$ satisfies the bound 
\begin{equation}\label{double coset measure}
m_{G_v}(K_v a_v K_v)\le c\, \Delta_{P_v}(a_v)\,\,, \mbox{   where $c>0$ is independent of $v$}. 
\end{equation}
Indeed, the estimate follows 
from the elementary proof 
of \cite[Theorem~4.1.1]{Si}, which also makes it plain that the constant $c$ is independent of $v$. 

To prove the main estimate we combine (\ref{Xi estimate}) and (\ref{double coset measure}), and for $p>4$ we obtain:
\begin{align*}
\int_{G_v} \Xi_{G_v}(g_v)^p\,dm_v(g_v) &=\sum_{a_v\in K_v\backslash G_v/K_v} \Xi_{G_v}(a_v)^p m_{G_v}(K_v a_v K_v)\\
&\le 1 +\sum_{a_v\in K_v\backslash G_v/K_v-K_v} (cc^p_\vre)\Delta_{P_v}(a_v)^{p(-1/2+\vre)+1} \\
&\le 1 +\sum_{a_v\in K_v\backslash G_v/K_v-K_v} (cc^p_\vre)\left(\prod_{\chi\in\Pi_v} |\chi (a_v)|_v\right)^{p(-1/2+\vre)+1} \\
&\le 1+\sum_{i_1,\ldots, i_r\in \mathbb{Z}_+, (i_1,\ldots, i_r)\ne
  0}(cc^p_\vre)q_v^{ (p(-1/2+\vre)+1)\sum_{j=1}^r i_j}\\
&=1+O_\vre\left(q_v^{p(-1/2+\vre)+1}\right).
\end{align*}
Since the Dedekind zeta function converges absolutely for $s>1$,
it also follows that  $\sum_{v\in V_f} q_v^{-s}<\infty$. 
Hence, $\prod_{v\notin V_0} \int_{G_v}  \Xi_{G_v}^p\,dm_v<\infty$,
as required.
\end{proof}
\begin{remark}
Regarding the estimate (\ref{double coset measure}), we note that an exact formula for the measure of a double coset was established for 
split simply connected groups in  \cite{Gr} as part of the discussion of the Satake transform. 
The fact that we may take $c=1+\frac{c_1}{q_v}$ is established in \cite[Lemma~6.11]{STBT2}
for adjoint groups, but only as a consequence of the computation of the integral of the local 
height function, which is less elementary. 
\end{remark}
\subsection{Analogue of the radial Kunze-Stein phenomenon on adele groups }
To complete the proof of Theorem \ref{rational points} we need, according to our general recipe from \S 3.1, to prove a stable quantitative mean ergodic
theorem for the Haar-uniform averages supported on the sets $B_T$. Our first step towards this goal is to establish a version the radial Kunze--Stein inequality for adele group, which is of considerable independent interest. 
\begin{theorem}\label{th:ks}
Let ${\sf G}$ be as in Theorem \ref{rational points}, and $G={\sf G}(\Adele)$. 
Let $q\in [1,4/3)$. Then for every absolutely continuous
bi-$K$-invariant probability measure $\beta$ such that $\|\beta\|_q<\infty$  and $f\in L^2(G)$,
$$
\norm{\beta * f}_2\le C_q \norm{\beta}_q\,\norm{f}_2.
%=m_G(\tilde B_T)^{-(1-1/p)}\|f\|_2
$$
\end{theorem}

\begin{proof}
Using that $G={\sf G}(\Adele)$ has an Iwasawa decomposition,
we can utilise Herz' argument as presented by Cowling \cite{Co2}.
The only difference is that the Harish-Chandra function $\Xi_{{\sf G}(\Adele)}$ is 
not in $L^{2+\vre}({\sf G}(\Adele))$, but in $L^{4+\vre}({\sf G}(\Adele))$, $\vre>0$ (see Proposition  \ref{p:xi}).
Therefore, this argument works only for $q<4/3$, which is the exponent dual to $4$.
\end{proof}

We can now establish a mean ergodic theorem with a rate for adele groups, as follows. 

\begin{corollary}\label{th:ad_mean}
Let ${\sf G}$ be as in Theorem  \ref{rational points}, $G={\sf G}(\Adele)$, and $B_T$ be the balls w.r.t. the height function. 
Let $G$ act on a standard Borel probability space $(X,\mu)$
and assume that the representation $\pi_X^0$ of $G$ on $L_0^2(X)$ is $L^{p+}$ for some $0< p < \infty$, or more generally, that $\left(\pi_X^0\right)^{\otimes n_e}$ is weakly contained in $\lambda_G$. 
Then the stable quantitative mean ergodic theorem holds in $L^2(X)$
for the Haar-uniform averages $\beta_T$ with the following estimate:
$$
\norm{\pi_X(\beta_T)f-\int_X fd\mu}_{L^2(X)}
\le C^\prime_\eta m_G(B_T)^{-(4n_e(p))^{-1}+\eta} \norm{f}_{L^2(X)},\quad \eta>0.
$$
\end{corollary}

\begin{proof}
As in the proof of Theorem \ref{semisimple mean}, $\left(\pi_X^0\right)^{\otimes n_e}$ is an $L^{2+}$-representation,
and
$$
\norm{\pi_X^0(\beta_T)}\le \norm{\lambda_{G}(\beta_T)}^{1/n_e}
$$
where $\lambda_G$ denotes the regular representation of $G$.

Let $\tilde B_T=KB_TK$ and $\tilde \beta_T$ denote the uniform averages supported on $\tilde B_T$.
By Theorem \ref{th:ks},
$$
\|\tilde \beta_T * f\|_2\le C_q m_G(\tilde B_T)^{-(1-1/q)}\|f\|_2.
$$
for every $q\in [1,4/3)$ and $f\in L^2(G)$. This implies that
$$
\norm{\lambda_{G}(\tilde\beta_T)}\le C^\prime_\eta  m_{G}(\tilde B_T)^{-1/4+\eta},\quad \eta>0.
$$
Since $\tilde B_T\subset B_{cT}$ for some $c>0$,  
it follows from the volume estimate in \cite[Section~4.3]{GMO} that $m_G(\tilde B_T) \le C m_G(B_T)$.
Hence,
$$
\norm{\lambda_{G}(\beta_T)}\le C \norm{\lambda_{G}(\tilde\beta_T)},
$$
and the claim follows. 
\end{proof}

\begin{proof}[Proof of  Theorem \ref{rational points}]
Every irreducible unitary representation $\pi$ of $G$
is unitarily equivalent to a restricted tensor
product $\pi=\bigotimes_v^\prime (\pi_v,\psi_v)$. Here 
$\pi_v$ is an irreducible unitary representation of the local 
group $G_v$ (see \cite{f}), and for almost all $v\in V$, the representation $\pi_v$ is spherical, namely the space of $G(\mathcal{O}_v)$-invariant vectors has dimension one, with  
$\psi_v$ denoting a unit vector invariant under 
$G(\mathcal{O}_v)$. This follows from \cite[Lemma 6.3]{Mo}, the fact that $\pi_v$ is irreducible, and the fact that  $(G(F_v),G(\mathcal{O}_v))$ is a Gelfand pair  (see also \cite[p. 733]{BC}).

It follows from the definition of the restricted 
tensor product $(\pi_v,\psi_v)$ w.r.t. to the $G(\mathcal{O}_v)$-invariant unit vectors $\psi_v$ (see e.g. \cite{Mo}),  that
there exists a canonical injective equivariant unitary map  
$$\left({\bigotimes_v}^\prime (\pi_v,\psi_v)\right)^{\otimes n} \longrightarrow {\bigotimes_v}^\prime(\pi_v^{\otimes n},\psi_v^{\otimes n})\,\,.$$

According to \cite{CL},
if $\pi$ is weakly contained in $L^2_0({\sf G}(\Adele)/{\sf G}(F))$,
then the local constituents $\pi_v$ are $L^p$-representations for some uniform $p$ independent of $v$.
Then for $n\ge p/2$ we have $\pi_v^{\otimes n}\subset \infty \cdot \lambda_{G_v}$, and in particular $\pi_v^{\otimes n}$ is weakly contained in the regular representation $\lambda_{G_v}$. Now according to \cite[Thm. 2]{BC}, if $\sigma_v$ is irreducible and weakly contained in $\lambda_{G_v}$ for each $v$, then $\bigotimes^\prime_v (\sigma_v,\phi_v)$ ($\phi_v$ a $G(\mathcal{O}_v)$-invariant unit vector) is weakly contained in  $\lambda_{{\sf G}(\Adele)}$. The proof of this fact however makes no use of the irreducibility assumption, and so is valid for  $\sigma_v=\pi_v^{\otimes n}$, $\phi_v=\psi_v^{\otimes n}$ as well. Hence 
$$\pi^{\otimes n}\cong \left({\bigotimes_v}^\prime \pi_v\right)^{\otimes n}\subseteq
{\bigotimes_v}^\prime\pi_v^{\otimes n}
%\subset \infty \cdot {\bigotimes_v}^\prime \lambda_{G_v}\cong \infty \cdot
\preceq  \lambda_{{\sf G}(\Adele)},$$ 
%where in the last equivalence we use the fact that
% the regular representation of 
%$G(\Adele)$ is the restricted tensor product of the regular 
%representations $\lambda_{G_v}$ 
%of the local constituents. This follows 
%from \cite{Bl}, and another elementary proof of weak containment, which suffices for our purposes,  is 
%due to \cite{BC}. 
%Now using Proposition \ref{p:xi} together with (\ref{eq:CHH}), we deduce that $\pi$ is $L^{p(\pi)}$ for some
%$0 < p(\pi)< \infty $. 
Since this argument applies to every irreducible representation weakly contained
in $L^2_0({\sf G}(\Adele)/{\sf G}(F))$, it follows that the $n$-th tensor power of 
$L^2_0({\sf G}(\Adele)/{\sf G}(F))$ is weakly contained in $\lambda_{{\sf G}(\Adele)/{\sf G}(F))}$ as well. 

%is an $L^{p(\pi)}$-representation as well.

To complete the proof of Theorem \ref{rational points}, we set $\mathcal{O}_\vre=\mathcal{O}^\infty_\vre\times W$
where
$$
\mathcal{O}^\infty_\vre =\{g=(g_v)\in G_\infty:\, H_v(g_v^{\pm 1}-id)<\vre, v\in V_\infty\},
$$
and $W\subset G_f$ is a compact open subgroup such that the height $H$ is $W$-bi-invariant.
Then $\mathcal{O}_\vre$ has local dimension at most $\dim(G_\infty)$.
The family $\{B_{e^t}\}$ is H\"older well-rounded with respect to
the neighbourhoods $\mathcal{O}_\vre$ (see \cite[Proposition~2.19(2)]{GO2}).
By Theorem \ref{th:ad_mean}, the stable quantitative mean ergodic theorem holds for the action 
of $G$ on $L^2_0({\sf G}(\Adele)/{\sf G}(F))$. Therefore Theorem \ref{rational points} follows from Theorem
\ref{error estimate}.
\end{proof}

\section{Angular distribution in symmetric spaces}\label{sec:sym}

\subsection{Definitions, notations, and statements of results}
Let $G$ be a (noncompact) connected semisimple Lie group with finite center.
Consider the Cartan decomposition
$$
G=KA^+K
$$
where $K$ is a maximal compact subgroup in $G$, and $A^+$ is a closed positive Weyl chamber in a Cartan subgroup
$A$ compatible with $K$. The $A^+$-component in this decomposition is unique, and the
$K$-components of regular elements are unique modulo $M$ where $M$ is the centraliser of $A$ in $K$.
An element $g=k_1ak_2$ is called $\delta$-regular (for some $\delta>0$)
if the distance of $a$ from the walls of the Weyl chamber
is at least $\delta$. Otherwise, the element is called $\delta$-singular.

We denote by $d$ the Cartan-Killing metric on the symmetric space $G/K$
and set $D_t=\{g\in G:\, d(gK, K)\le t\}$. For $\Phi,\Psi\subset K$, we consider bisectors:
$$
D_t(\Phi,\Psi)=\{k_1ak_2:\;\;\; k_1\in\Phi,\;\; a\in A^+,\; d(aK,K)\le t,\;\; k_2\in\Psi\}.
$$
We are interested in the distribution of lattice points with respect to bisectors.
The main term in this problem was investigated in \cite{GO} and \cite{GOS2}, but the issue of rates
in the asymptotic estimates was not addressed.

We denote by $D^\delta_t$ the subset of $\delta$-singular elements of $D_t$.
It will be crucial that most of the volume is concentrated in the interior of the Weyl chamber, i.e.,
there exists $\zeta_0>0$ such that for every $\delta>0$,
\begin{equation}\label{eq:beta}
\vol(D_t^\delta)=O_{\delta,\eta}\left(\vol(D_t)^{1-\zeta_0+\eta}\right),\quad \eta>0,
\end{equation}
Let us note that  \eqref{eq:beta2} below gives the precise value of $\zeta_0$ in terms of the root system of $G$. The lower order of magnitude of the volume of the neighbourhood $D_t^\delta$ of the singular set  is the main difference between the bisectors on Riemannian symmetric spaces discussed in the present section, and bisectors in more general affine symmetric spaces considered in Section \ref{sec:affine}.

We fix a base of neighbourhoods $\mathcal{O}_\vre$ of identity in $G$ with respect to a
(right) invariant Riemannian metric. A measurable subset $\Phi$ of a homogeneous space of $K$ is called Lipschitz
well-rounded if 
$$
\vol\left((\mathcal{O}_\vre\cap K)\Phi-\cap_{u\in\mathcal{O}_\vre\cap K} u\Phi\right)\ll_\Phi \vre
$$
for $\vre\in (0,\vre_1)$. 
For example, it is easy to check that 
balls with respect to an invariant Riemannian metric are Lipschitz well-rounded.
Note that this notion does not depend on a choice of a Riemannian metric.

\begin{theorem}\label{th:sector}
Let $\Gamma$ be a lattice in $G$ such that  the representation $\pi_{G/\Gamma}^0$ in $L_0^2(G/\Gamma)$ is $L^{p+}$ for some $p>0$,
$\Phi$ a Lipschitz
well-rounded subset of $K/M$ with positive measure, and $\Psi$ a Lipschitz well-rounded subset of
$M\backslash K$  with positive measure.
Then
$$
|\Gamma\cap D_t(\Phi,\Psi) |=\frac{\vol(D_t(\Phi,\Psi))}{\vol(G/\Gamma)}+O_{\Phi,\Psi,\eta}(\vol(D_t)^{1-\zeta+\eta}),\quad\eta>0,
$$
where $\zeta=\min\{\zeta_0,(2n_e(p))^{-1}(1+\dim G)^{-1}\}$.
Moreover, this estimate is uniform over all lattices such that
the representation $L_0^2(G/\Gamma)$ is $L^{p+}$ and $\vre_0(e,\Gamma)\ge \vre_0$
with fixed $\vre_0>0$.
\end{theorem}

We also state a version of this theorem in the language of test-functions.
Since it is essentially equivalent
to Theorem \ref{th:sector}, we only give a proof of Theorem \ref{th:sector}.
For an element $g\in G$, we write its Cartan decomposition as $g=k_1(g)a(g)k_2(g)$.

\begin{theorem}\label{th:sector2}
Let $\Gamma$ be a lattice in $G$ such that  the representation $\pi_{G/\Gamma}^0$ in $L_0^2(G/\Gamma)$ is $L^{p+}$ for some $p>0$,
$\phi_1$ a Lipschitz function on $K/M$, and $\phi_2$ a Lipschitz function on $M\backslash K$.
Then
\begin{align*}
\sum_{\gamma\in \Gamma\cap D_t}\phi_1(k_1(\gamma))\phi_2(k_2(\gamma)) =&\frac{\vol(D_t)}{\vol(G/\Gamma)}
\left(\int_{K/M} \phi_1 \, dk\right)\left(\int_{M\backslash K} \phi_2 \,dk\right)\\
&+
O_{\phi_1,\phi_2,\eta}(\vol(D_t)^{1-\zeta+\eta}),\quad\eta>0,
\end{align*}
where $\zeta$ is as in Theorem \ref{th:sector}.
\end{theorem}

The proof of Theorem \ref{th:sector} is based on Theorem \ref{error estimate}. According to the recipe of \S 3.1, we must verify the spectral condition for the averages  supported on $D_t(\Phi,\Psi)$, and then show that they are H\"older well-rounded. The spectral estimate is an immediate consequence of Theorem \ref{semisimple mean}. The main issue here is the regularity of the sets, and we now undertake the task of showing 
that they are in fact Lipschitz well-rounded.    
\subsection{Quantitative wave front lemma}

It was shown in \cite{N2} that the components of the Cartan decomposition $G=KA^+K$
of regular elements vary continuously under small 
perturbations. The proof is very simple and based on the proof of the wave-front Lemma given in \cite[Lemma 5.11]{EMM}
(see also \cite[Theorem 3.1]{EM}).  We apply this argument to show that the variation of the Cartan components is Lipschitz.
See also \cite{GOS2} for a different argument.

%We fix a (right) invariant Riemannian metric on $G$ and denote by $\mathcal{O}_\vre$
%the $\vre$-neighbourhood of identity in $G$.
For $\delta>0$, we denote by $\tilde A^\delta$ the subset of $A^+$ consisting of elements with distance
$\ge \delta$ from the walls. 

\begin{proposition}\label{wave front}{\bf Effective Cartan decomposition.}
Let $\delta>0$. There exist $\vre_0,\ell_0>0$ such that for every $g=k_1ak_2\in K\tilde A^\delta K$
and $\vre\in (0,\vre_0)$,
$$
\mathcal{O}_{\vre}\, g\, \mathcal{O}_\vre\subset (\mathcal{O}_{\ell_0\vre}\cap K) k_1M\,
(\mathcal{O}_{\ell_0\vre}\cap A)a
\,k_2(\mathcal{O}_{\ell_0\vre}\cap K).
$$
\end{proposition}

\begin{proof}
Using that $K$ is compact, it is easy to reduce the proof to showing that 
$$a\mathcal{O}_\vre\subset (\mathcal{O}_{\ell_0\vre}\cap K)M\,
(\mathcal{O}_{\ell_0\vre}\cap A)a \,(\mathcal{O}_{\ell_0\vre}\cap K).$$
Since the proof of \cite[Lemma~5.11]{EMM} implies that
$$a\mathcal{O}_\vre\subset K\,
(\mathcal{O}_{\ell_0\vre}\cap A)a \,K$$
for some $\ell_0>0$, it remains to analyse behaviour of the $K$-components. 

Let $P$ be the standard parabolic subgroup,
i.e., $P=MAU$ where $M$ is the centraliser of $A$ in $K$ and $U$ is
the subgroup generated by positive root subgroups.
There exists a real representation $G\to\hbox{GL}(V)$ such that
for some vector $e_1\in V$, its projective stabiliser is $P$ \cite[Theorem~4.29]{gjt}. 
Without loss of generality, we may assume that $V=\left<G e_1\right>$.
One can choose a Euclidean structure on $V$
such that $K$ consists of orthogonal matrices, $A$ consists of self-adjoint
matrices, and $\|e_1\|=1$. We fix an orthonormal basis $\{e_i\}$ of eigenvectors of $A$.
If $ae_1=e^{\lambda(\log a)}e_1$ for $\lambda\in\hbox{Lie}(A)^*$,
then the weights of $e_i$'s are of the form $\lambda-\alpha_i$ where $\alpha_i$
is a positive linear combination of positive roots. In particular, it follows that for
every $a\in A^+$, $\|a\|=e^{\lambda(\log a)}$.

Let $a\in \tilde A^\delta$, $g\in\mathcal{O}_\vre$ and $ag=k_1bk_2$ for $k_1,k_2\in K$ and $b\in A^+$.
We will show that for some $c>0$, we have $\|k_1e_1-e_1\|<c\vre$ and $\|k_2^{-1}e_1-e_1\|<c\vre$.
Since $P\cap K=M$, this implies that $k_1 \in (\mathcal{O}_{c'\vre}\cap K)M$ and
$k_2\in M(\mathcal{O}_{c'\vre}\cap K)$ for some  $c'>0$, as required.

We have
\begin{align}\label{eq:l1}
e^{\lambda(\log b)}=\inn {be_1,e_1}=\inn{k_1^{-1} agk_2^{-1}e_1,e_1}\le \|agk_2^{-1}e_1\|.
\end{align}
Writing $gk_2^{-1}e_1=\sum_i u_i e_i$ with $u_i\in\mathbb{R}$, we get
$$
\|agk_2^{-1}e_1\|^2=e^{2\lambda(\log a)}\sum_i e^{-2\alpha_i(\log a)}u_i^2\le
e^{2\lambda(\log a)}\left(u_1^2+\sum_{i>1} e^{-c_1\delta }u_i^2\right)
$$
for some $c_1>0$. Since $g\in\mathcal{O}_\vre$, 
\begin{equation}\label{eq:l2}
\|gk_2^{-1}e_1\|^2\le 1+c_2\vre
\end{equation}
for some $c_2>0$, and
\begin{equation}\label{eq:l3}
\|agk_2^{-1}e_1\|^2\le e^{2\lambda(\log a)}\left(u_1^2+e^{-c_1\delta }(1+c_2\vre-u_1^2)\right).
\end{equation}
Since $\|b\|^2=\|ag\|^2\ge (1-c_3\vre)\|a\|^2$ for some $c_3>0$, it follows that 
$e^{2\lambda(\log b)-2\lambda(\log a)}\ge 1-c_3\vre$. Hence,
combining \eqref{eq:l1} and \eqref{eq:l3}, we get
$$
u_1^2+e^{-c_1\delta }(1+c_2\vre-u_1^2)\ge  1-c_3\vre,
$$
and
$$
u_1^2\ge \frac{1-c_3\vre - e^{-c_1\delta }(1+c_2\vre)}{1-e^{-c_1\delta }}.
$$
This shows that $|u_1|= 1+O_\delta(\vre)$.
Then by \eqref{eq:l2},  $\|gk_2^{-1}e_1- e_1\|=O_\delta(\vre)$.
Since $g\in\mathcal{O}_\vre$, it follows that $\|k_2^{-1}e_1- e_1\|=O_\delta(\vre)$ as well.

The proof that $\|k_1e_1- e_1\|=O_\delta(\vre)$ is similar.
\end{proof}

\begin{proposition}\label{bisectors Lip}
Let $\Phi$ be a Lipschitz well-rounded subset of $K/M$ with positive measure
and $\Psi$ a  Lipschitz well-rounded subset of $M\backslash K$ with positive measure.
Then for every $\delta>0$, the family of sets 
$$
\tilde D^\delta_t(\Phi,\Psi):=\{k_1ak_2:\;\;\; k_1\in\Phi,\;\; a\in \tilde A^\delta,\; d(aK,K)\le t,\;\; k_2\in\Psi\}.
$$
is Lipschitz well-rounded.
\end{proposition}

Before we start the proof, we recall some facts about volumes.
The Haar measure in $KA^+K$-coordinates is given by $dk_1\,\xi(a)da\, dk_2$ where
$dk_1$, $da$, $dk_2$ are Haar measures on the components and
\begin{equation}\label{eq:xi}
\xi(a)=\prod_{\alpha\in\Sigma^+} (e^{\alpha(\log a)}-e^{-\alpha(\log a)})
\end{equation}
(here $\Sigma^+$ is the set of positive roots). We denote by $2\rho$ the sum of positive roots with multiplicities.
It is well-known that for  every $\eta>0$ and $c_\eta>1$,
\begin{equation}\label{eq:D_t}
c^{-1}_\eta\, e^{(\alpha-\eta)t}\le \vol(D_t)\le c_\eta\, e^{(\alpha+\eta)t}
\end{equation}
where $\alpha=\max\{2\rho(\log a): a\in A^+\cap D_1\}$.\footnote{In fact, the exact asymptotic is known,
  but we do not need it here.}
We also set 
$\alpha_0=\max\{2\rho(\log a): a\in \hbox{walls}(A^+)\cap D_1\}$.
Since the balls are strictly convex, $\alpha_0<\alpha$. 
Hence, using that  $\xi(a)\le e^{2\rho(\log a)}$, we deduce from (\ref{eq:D_t}) that
\begin{equation}\label{eq:beta2}
\vol(D_t^\delta)=O_{\delta,\eta}\left(\vol(D_t)^{\alpha_0/\alpha+\eta}\right),\quad \eta>0.
\end{equation}

\begin{proof}[Proof of Proposition \ref{bisectors Lip}]
By Proposition \ref{wave front},
\begin{align*}
\mathcal{O}_\vre \tilde D^\delta_t(\Phi,\Psi) \mathcal{O}_\vre\subset
\tilde D^{\delta-\ell_0\vre}_{t+2\vre}(\Phi^+_\vre,\Psi^+_\vre)\quad\hbox{and}\quad
\bigcap_{u,v\in\mathcal{O}_\vre} u \tilde D^\delta_t(\Phi,\Psi) v \supset
\tilde D^{\delta+\ell_0\vre}_{t-2\vre}(\Phi^-_\vre,\Psi^-_\vre)
\end{align*}
where $\Phi_\vre^+=(\mathcal{O}_\vre\cap K)\Phi$, $\Psi^+_\vre=\Psi(\mathcal{O}_\vre\cap K)$,
$\Phi^-_\vre=\cap_{u\in\mathcal{O}_\vre\cap K} u\Phi$, 
$\Psi^-_\vre=\cap_{u\in \mathcal{O}_\vre\cap K} \Psi u$.
Hence, it remains to estimate
\begin{align*}
&\vol\left(\tilde D^{\delta-\ell_0\vre}_{t+2\vre}(\Phi^+_\vre,\Psi^+_\vre)
-\tilde D^{\delta+\ell_0\vre}_{t-2\vre}(\Phi^-_\vre,\Psi^-_\vre)\right)\\
\le&
\vol\left(\tilde D^{\delta-\ell_0\vre}_{t+2\vre}(\Phi^+_\vre-\Phi^-_\vre,\Psi^+_\vre)\right)
+\vol\left(\tilde D^{\delta-\ell_0\vre}_{t+2\vre}(\Phi^+_\vre,\Psi^+_\vre-\Psi^-_\vre)\right)\\
&+\vol\left(\tilde D^{\delta-\ell_0\vre}_{t+2\vre}-\tilde D^{\delta+\ell_0\vre}_{t-2\vre} \right).
\end{align*}
Since the sets $\Phi$ and $\Psi$ are Lipschitz well-rounded, the first and the second terms are 
$O\left(\vre \vol(D_{t+2\vre})\right)$. We estimate the last term by
\begin{align*}
\vol\left(\tilde D^{\delta-\ell_0\vre}_{t+2\vre}-\tilde D^{\delta+\ell_0\vre}_{t-2\vre} \right)\le
\vol\left(D_{t+2\vre}-D_{t-2\vre}\right)
+\vol\left(\tilde D^{\delta-\ell_0\vre}_{t+2\vre}-\tilde D^{\delta+\ell_0\vre}_{t+2\vre}\right).
\end{align*}
It was shown in \cite[Proposition 7.1]{GN} that the function $t\mapsto \log \vol(D_t)$ is
uniformly locally Lipschitz. It follows from the formula for the Haar measure, (\ref{eq:xi}) and
(\ref{eq:D_t}) that
$$
\vol\left(\tilde D^{\delta-\ell_0\vre}_{t+2\vre}-\tilde D^{\delta+\ell_0\vre}_{t+2\vre}\right)
\ll \vre\, t^{\dim A^+-1} e^{\alpha_0 t}\ll \vre\,\vol(D_{t+2\vre}).
$$
We conclude that
$$
\vol\left(\tilde D^{\delta-\ell_0\vre}_{t+2\vre}(\Phi^+_\vre,\Psi^+_\vre)
-\tilde
D^{\delta+\ell_0\vre}_{t-2\vre}(\Phi^-_\vre,\Psi^-_\vre)\right)=O\left(\vre\vol(D_{t+2\vre})\right).
$$
Finally, the claim follows from (\ref{eq:beta2}) and 
the Lipschitz property of the function $t\mapsto \log \vol(D_t)$.
\end{proof}

\begin{proof}[Proof of Theorem \ref{th:sector}]
By Proposition \ref{bisectors Lip}, the family $\{\tilde D_t^\delta(\Phi,\Psi)\}$ is Lipschitz
well-rounded, and by Theorem \ref{semisimple mean}, the corresponding averages satisfy the stable 
quantitative mean ergodic theorem. Hence, by Theorem \ref{error estimate},
\begin{equation*}
|\Gamma\cap \tilde D^\delta_t(\Phi,\Psi) |=\frac{\vol(\tilde D^\delta_t(\Phi,\Psi))}{\vol(G/\Gamma)}+O_{\Phi,\Psi,\eta}\left(\vol(\tilde
D^\delta_t(\Phi,\Psi))^{1-\zeta_1+\eta}\right),\quad\eta>0,
\end{equation*}
where $\zeta_1=(2n_e(p))^{-1}(1+\dim G)^{-1}$.  Hence, by (\ref{eq:beta2}),
\begin{equation}\label{eq:s1}
|\Gamma\cap \tilde D^\delta_t(\Phi,\Psi) |=\frac{\vol(\tilde D^\delta_t(\Phi,\Psi))}{{\vol(G/\Gamma)}}+O_{\Phi,\Psi,\eta}\left(\vol(
D_t)^{1-\zeta_1+\eta}\right),\quad\eta>0.
\end{equation}
It remains to
estimate the number of lattice points in the singular set, namely $|\Gamma\cap D^\delta_t(\Phi,\Psi) |$. 
To that end, 
fix a symmetric neighbourhood $\mathcal{O}_\omega$ of
identity such 
that $\Gamma\cap \mathcal{O}^2_\omega=\{e\}$. Then 
$$
|\Gamma\cap D^\delta_t(\Phi,\Psi) |\le \frac{\vol(\mathcal{O}_\omega
  D^\delta_t)}{\vol(\mathcal{O}_\omega)},
$$
and by Proposition \ref{wave front},
$$
\vol(\mathcal{O}_\omega D^\delta_t)\le \vol(D^{\delta-\ell_0\omega}_{t+\omega}).
$$
Hence, by (\ref{eq:beta}) and
the Lipschitz property of the function $t\mapsto \log \vol(D_t)$,
\begin{equation}\label{eq:s2}
|\Gamma\cap D^\delta_t(\Phi,\Psi) |=O_{\delta,\eta}\left(\vol(D_t)^{1-\zeta_0+\eta}\right).
\end{equation}
Finally, combining \eqref{eq:s1} and \eqref{eq:s2}, we deduce the claim.
\end{proof}

\section{Lattice points on affine symmetric varieties}\label{sec:affine}

In the present section we consider the lattice point counting problem in subsets of a connected semisimple Lie group arising from  sectors in affine symmetric spaces, and give an explicit quantitative estimate of the error in all cases. This result gives an explicit quantitative solution of the lattice point counting problem on $G/H$ itself whenever $\Gamma\cap H$ is co-compact in $H$. 
%The main term in the asymptotics was established independently in \cite{EM} and in \cite{DRS}, where an error estimate was also established in certain special cases. 

We note also that our solution will be uniform over all subgroups of finite index in the lattice, provided they all admit a uniform spectral gap, namely satisfy property $\tau$. This uniformity property plays a crucial role in establishing the existence of the right order of magnitude of almost prime points on the algebraic variety $G/H$, provided $G$ and $H$ are defined over $\QQ$ and $\Gamma$ is the lattice of integral points. This and other applications of the solution of the lattice point counting problem will be 
elaborated elsewhere.   

\subsection{Notation, definitions and statements of results}
Throughout the present section, we let $G$ be a connected semisimple Lie group with finite center
and $H$ is a closed symmetric subgroup of $G$ (that is, the Lie algebra
of $H$ is the set of  fixed points of an involution $\sigma$).
Let $K$ be a maximal compact subgroup compatible with $H$ (this means that
the involution $\theta$ corresponding to $K$ commutes with $\sigma$). 
Let $G^{\sigma\theta}$ be the subgroup of fixed points of $\sigma\theta$ and
$A$ a Cartan subgroup of $G^{\sigma\theta}$ compatible with $K\cap G^{\sigma\theta}$.
The group $A$ is equipped with a root system
(for the action of $A$ on $G^{\sigma\theta}$). 
We fix a system of positive roots and denote by $A^+$ a closed positive Weyl chamber in $A$.
Then we have the Cartan decomposition
$$
G=KA^+H.
$$
We say that an element $a\in A$ is $\delta$-regular
if the distance of $a$ from the boundary of $A^+$ is at least $\delta$, and regular if it $\delta$-regular for some $\delta > 0$.
More generally, an element $g\in G$ is called $\delta$-regular if its $A$-component is $\delta$-regular.
Note that the $A^+$-component of an element is uniquely defined, and
the $K$- and $H$-components of a regular element are uniquely defined modulo the subgroup $M$
which is the centraliser of $A$ in $K\cap H$.
We refer to \cite[Ch.~7]{sch1} and \cite[Part~II]{sch2} for basic facts about affine symmetric spaces.

Let $\rho:G\to \hbox{GL}(V)$ be a representation of $G$, and $v_0\in V$ 
be such that $\hbox{Stab}_G(v_0)=H$. We fix a norm on $V$ and define
$$
S_t=\{g\in G:\, \log \|gv_0\|\le t\}\quad\hbox{and}\quad B_t=\{g\in G/H:\, \log \|gv_0\|\le t\}.
$$
For sets $\Phi\subset K$ and $\Psi\subset H$, define
$$
S_t(\Phi,\Psi)=S_t\cap \Phi A^+\Psi.
$$
We compute the asymptotics of the number of lattice points in $S_t(\Phi,\Psi)$.
Our argument is based on the effective version of the Cartan decomposition, which draws on some arguments in \cite{GOS2}.
The main term in the asymptotic of lattice points $S_t(\Phi,\Psi)$ was also computed in \cite{GOS2}, but the problem of rates was not addressed. As noted in \S 1.3, the method of mixing used there generally gives an error term inferior to the one established below.

\begin{theorem}\label{c:s22}
Let $\Gamma$ be a lattice in $G$ such that the representation $\pi_{G/\Gamma}^0$ in $L_0^2(G/\Gamma)$ is $L^{p+}$ for some $p>0$.
Let $\Phi$ be Lipschitz well-rounded subset of $K/M$ with positive measure, 
and $\Psi$ a bounded Lipschitz well-rounded subset of $M\backslash H$ with positive measure.
Then
$$
|\Gamma\cap S_t(\Phi,\Psi) |=\frac{\vol(S_t(\Phi,\Psi))}{\vol(G/\Gamma)}+O_{\Phi,\Psi,\eta}(\vol(S_t(\Phi,\Psi))^{1-\zeta+\eta}),\quad \eta>0,
$$
where $\zeta=(2n_e(p))^{-1}(1+3\dim G)^{-1}$.
Moreover, this estimate is uniform over all lattices such that
the representation $L_0^2(G/\Gamma)$ is $L^{p+}$ and $\vre_0(e,\Gamma)\ge \vre_0$
with fixed $\vre_0>0$.
\end{theorem}

\begin{remark}
Let us comment on the difference between Theorem \ref{th:sector} and Theorem \ref{c:s2}.
While the balls on symmetric spaces are defined with respect to the
Cartan-Killing metric, the ball in affine symmetric spaces are defined with respect
to a norm. In the later case, an analogue of estimate (\ref{eq:beta}) fails,
and we will need a more elaborate argument to prove well-roundedness. 
As a result, while the family $D_t(\Phi,\Psi)$ in Riemannian symmetric spaces is shown to be Lipschitz well-rounded, we can only show that the family $S_t(\Phi,\Psi)$ in affine symmetric spaces is H\"older well-rounded with exponent $1/3$. 
\end{remark}

Let $\Gamma$ be a lattice in $G$ such that
$\Gamma\cap H$ is cocompact in $H$. Then the orbit $\Gamma v_0$ is discrete.
Given $\Phi\subset K/M$, we are interested in the effective asymptotic of 
$$
\Gamma v_0\cap \{v\in \Phi A^+v_0 :\, \log \|v\|\le t\}.
$$

\begin{corollary}\label{c:s2}
Let  $\Phi$ be a Lipschitz well-rounded subset of $K/M$ with positive measure.
Then for $B_t(\Phi)=\Phi A^+H\cap B_t$,
$$
|\Gamma H\cap B_t(\Phi) |=\frac{\vol(H/(H\cap\Gamma))}{\vol(G/\Gamma)}\vol(B_t(\Phi))+O_{\Phi,\eta}(\vol(B_t(\Phi))^{1-\zeta+\eta}),\quad\eta>0,
$$
where $\zeta$ is as in Theorem \ref{c:s22}.
\end{corollary}

As noted already, the crucial step in the proof is to show that the family of sets
$S_t(\Phi,\Psi)$ is
H\"older well-rounded.

\begin{proposition}\label{th:sec_sym}
Let $\Phi$ be a Lipschitz well-rounded subset of $K/M$ with positive measure and
$\Psi$ a bounded Lipschitz well-rounded subset of $M\backslash H$ with positive measure.
Then the family of sets $S_t(\Phi,\Psi)$ is H\"older well-rounded with exponent $1/3$.
\end{proposition}

\begin{remark}\label{r:sec}
More generally, it will be clear from the proof that any family of measurable subsets 
of $S_t(\Phi,\Psi)$ that contain all regular elements of $S_t(\Phi,\Psi)$ is H\"older
well-rounded with exponent $1/3$. This remark will be used below. 

\end{remark}

\subsection{Holder well-roundedness of sector averages}
In preparation for the proof of Proposition \ref{th:sec_sym}, 
we note the following quantitative result on the Lipschitz property of the Cartan decomposition,
which is based on arguments appearing in \cite{GOS2}. Let $\mathcal{O}_\vre$ denote
the $\vre$-neighbourhood of identity in $G$ with respect to a (right) invariant
Riemannian metric in $G$, so that 
$\mathcal{O}_{\vre_1}\cdot \mathcal{O}_{\vre_2}=\mathcal{O}_{\vre_1+\vre_2}$.

\begin{proposition}\label{th:wave}
Let $\delta\in (0,1)$. There exist $\vre_0,\ell_0>0$ such that for every $\delta$-regular $a\in A$
and $\vre\in (0,\vre_0)$, we have
$$
\mathcal{O}_\vre\, a\, \mathcal{O}_\vre\subset (\mathcal{O}_{\ell_0\vre}\cap K)\,
(\mathcal{O}_{\ell_0\vre}\cap A)a \, (\mathcal{O}_{\ell_0\vre}\cap H).
$$
Moreover,  
$$
\hbox{$\vre_0\gg \delta^3$ and $\ell_0\ll\delta^{-2}$}\quad \hbox{as $\delta\to 0^+$.}
$$
\end{proposition}

\begin{proof}
It was shown in \cite[Theorem~4.1]{GOS2}, that there exist $\vre_0,\ell_0>0$
such that for every $\delta$-regular $a\in A$ and $\vre\in (0,\vre_0)$,
$$
\mathcal{O}_\vre a \subset (\mathcal{O}_{\ell_0\vre}\cap K)
(\mathcal{O}_{\ell_0\vre}\cap A)a (\mathcal{O}_{\ell_0\vre}\cap H).
$$
The estimates $\vre_0\gg\delta^2$ and $\ell_0\ll\delta^{-1}$ 
as $\delta\to 0^+$ can be extracted from the proof.

A straight-forward modification of the arguments in the proof in \cite[Theorem~4.1]{GOS2} also gives that
$$
a\mathcal{O}_\vre \subset (\mathcal{O}_{\ell_0\vre}\cap K)
(\mathcal{O}_{\ell_0\vre}\cap A)a (\mathcal{O}_{\ell_0\vre}\cap H).
$$
There exists $c>0$ such that for every $\vre\in (0,\vre_0/(2\ell_0+1))$,
\begin{align*}
\mathcal{O}_\vre a \mathcal{O}_\vre &\subset (\mathcal{O}_{\ell_0\vre}\cap K)
(\mathcal{O}_{\ell_0\vre}\cap A)a (\mathcal{O}_{\ell_0\vre}\cap H)\cdot\mathcal{O}_\vre\\
&\subset (\mathcal{O}_{\ell_0\vre}\cap K)a\mathcal{O}_{(2\ell_0+1)\vre}\\
&\subset (\mathcal{O}_{\ell_0\vre}\cap K)(\mathcal{O}_{\ell_0(2\ell_0+1)\vre}\cap K)
(\mathcal{O}_{\ell_0(2\ell_0+1)\vre}\cap A)a (\mathcal{O}_{\ell_0(2\ell_0+1)\vre}\cap H)\\
&\subset (\mathcal{O}_{\ell_0(2\ell_0+2)\vre}\cap K)
(\mathcal{O}_{\ell_0(2\ell_0+1)\vre}\cap A)a (\mathcal{O}_{\ell_0(2\ell_0+1)\vre}\cap H).
\end{align*}
This implies the proposition.
\end{proof}

Fix a positive Weyl chamber $A^{++}$ in $A$
for the action of $A$ on $G$. Note that this Weyl chamber is smaller than $A^+$,
and $A^+$ is finite union of chambers of the form $A^{++}$.
Let $Z$ denote the centraliser of $A$ in $G$, and
$U^+$ and $U^-$ are expanding and contracting subgroups corresponding to $A^{++}$.

\begin{proposition}\label{l:pm}
There exist $c,\vre_0>0$ such that for every $a\in A^{++}$ and $\vre\in (0,\vre_0)$,
$$
\mathcal{O}_\vre a \mathcal{O}_\vre\subset (U^-\cap \mathcal{O}_{c\vre})
 (Z\cap \mathcal{O}_{c\vre})a
 (U^+\cap \mathcal{O}_{c\vre}).
$$
\end{proposition}

\begin{proof}
Since the product map $U^-\times Z\times U^+\to G$ is a diffeomorphism
in a neighbourhood of identity, there exist $c_0,\vre_0>0$ such that for every
$\vre\in (0,\vre_0)$,
\begin{align*}
\mathcal{O}_\vre &\subset (U^-\cap \mathcal{O}_{c_0\vre})
 (Z\cap \mathcal{O}_{c_0\vre})
 (U^+\cap \mathcal{O}_{c_0\vre}).
\end{align*}
There exist $c_1,\vre_1>0$ such that for every $a\in A^{++}$ and $\vre\in (0,\vre_1)$,
\begin{align*}
a^{-1}(U^+\cap \mathcal{O}_\vre) a &\subset (U^+\cap \mathcal{O}_{c_1\vre}),\\
a(U^-\cap \mathcal{O}_\vre) a^{-1} &\subset (U^-\cap \mathcal{O}_{c_1\vre}).
\end{align*}
Hence, it follows that
\begin{align*}
\mathcal{O}_\vre a \mathcal{O}_\vre &\subset (U^-\cap \mathcal{O}_{c_0\vre})a
\cdot a^{-1}(Z\cap \mathcal{O}_{c_0\vre})(U^+\cap
\mathcal{O}_{c_0\vre})a\cdot\mathcal{O}_\vre\\
&\subset (U^-\cap \mathcal{O}_{c_0\vre})a
\cdot \mathcal{O}_{(c_1c_0+c_0+1)\vre},
\end{align*}
and
\begin{align*}
a\cdot  \mathcal{O}_\vre &\subset a(U^-\cap \mathcal{O}_{c_0\vre})(Z\cap
\mathcal{O}_{c_0\vre})a^{-1}\cdot a (U^+\cap \mathcal{O}_{c_0\vre})\\
&\subset (U^-\cap \mathcal{O}_{c_1c_0\vre})(Z\cap
\mathcal{O}_{c_0\vre}) a (U^+\cap \mathcal{O}_{c_0\vre}).
\end{align*}
Now the proposition follows from the last two estimates.
\end{proof}

\begin{proof}[Proof of Proposition \ref{th:sec_sym}]
Let $\vre,\delta,\ell>0$ be such that for every $\delta$-regular $a\in A$,
$$
\mathcal{O}_\vre a \mathcal{O}_\vre\subset (\mathcal{O}_{\ell\vre}\cap K)
(\mathcal{O}_{\ell\vre}\cap A)a (\mathcal{O}_{\ell\vre}\cap H).
$$
By Proposition \ref{th:wave}, for every small $\vre>0$, such $\delta$ and $\ell$ exist,
and we have 
\begin{equation}\label{eq:dl}
\delta=O(\vre^{1/3})\quad\hbox{and}\quad \ell=O(\vre^{-2/3}).
\end{equation}

Let $S_t^{\delta}(\Phi,\Psi)$ denote the subset of $\delta$-singular elements of $S_t(\Phi,\Psi)$.
We claim that
\begin{equation}\label{eq:11}
\vol(\mathcal{O}_\vre S^\delta_t(\Phi,\Psi)\mathcal{O}_\vre)\ll \vre^{1/3} \vol(S_t(\Phi,\Psi)).
\end{equation}
Decomposing $A^+$ into a union of the Weyl chambers $A^{++}$, it suffices to prove this estimate
for a subset of $S^\delta_t(\Phi,\Psi)$ with $A$-component contained in $A^{++}$.
Let $A^\delta_t$ be the subset of $S_t\cap A^{++}$ consisting of $\delta$-singular elements.
There exists $c=c(\Phi)>0$ such that
$$
S^\delta_t(\Phi,\Psi)\cap KA^{++}H\subset \Phi A^\delta_{t+c}\Psi.
$$
Let $k_1,\ldots, k_I\in \Phi$ be an $\vre$-net in $\Phi\subset  K/M$ such that $I=O(\vre^{-(\dim K-\dim M)})$
and $h_1,\ldots, h_J\in H$ an $\vre$-net in $M\Psi$ such that $J=O(\vre^{-\dim H})$.
Then for some $c_1>0$,
\begin{align*}
\mathcal{O}_\vre\cdot (\Phi A^\delta_{t+c}\Psi)\cdot \mathcal{O}_\vre
\subset \bigcup_{i,j} \mathcal{O}_{2\vre} k_i A^\delta_{t+c} h_j\mathcal{O}_{2\vre} 
\subset \bigcup_{i,j} k_i \mathcal{O}_{c_1\vre}
A^\delta_{t+c} \mathcal{O}_{c_1\vre}h_j.
\end{align*}
Hence,
\begin{equation}\label{eq:sing}
\vol(\mathcal{O}_\vre \cdot (\Phi A^\delta_{t+c}\Psi)\cdot \mathcal{O}_\vre)
\le IJ \vol(\mathcal{O}_{c_1\vre}
A^\delta_{t+c} \mathcal{O}_{c_1\vre}).
\end{equation}
%We fix a Weyl chamber $A^+$ in $A$ and set $A^{+,\delta}_{t}=A^+\cap A^{\delta}_{t}$.
%Since $A$ can be decomposed as a union of Weyl chambers, it suffices to estimate
%the volume of $\mathcal{O}_{(c_1+1)\vre} A^{+,\delta}_{t+c} \mathcal{O}_{(c_1+1)\vre}$.
By Proposition \ref{l:pm}, for some $c_2>0$,
$$
\mathcal{O}_{c_1\vre}
A^{\delta}_{t+c} \mathcal{O}_{c_1\vre}
\subset (U^-\cap \mathcal{O}_{c_2\vre})
 (Z\cap \mathcal{O}_{c_2\vre}) A^{\delta}_{t+c}
 (U^+\cap \mathcal{O}_{c_2\vre}).
$$
Since $Z$ is reductive and $A$ lies in the center of $Z$, there exist a local Lie
subgroup $Z'$ complementary to $A$ and for some $c_3>0$,
$$
(Z\cap \mathcal{O}_{c_2\vre})\subset (Z'\cap \mathcal{O}_{c_3\vre})(A\cap
\mathcal{O}_{c_3\vre}).
$$
Therefore, for some $c_4>0$,
$$
\mathcal{O}_{c_1\vre}
A^{\delta}_{t+c} \mathcal{O}_{c_1\vre}\subset
(U^-\cap \mathcal{O}_{c_2\vre})
 (Z'\cap \mathcal{O}_{c_3\vre}) A^{\delta+c_3\vre}_{t+c_4}
 (U^+\cap \mathcal{O}_{c_2\vre}).
$$
The Haar measure on $G$ with respect to $U^-Z'AU^+$-coordinates is given by
$$
\det(\Ad(a)|_{U^+})du^-dz'dadu^+.
$$
Therefore,
$$
\vol(\mathcal{O}_{c_1\vre}
A^{\delta}_{t+c} \mathcal{O}_{c_1\vre})
\ll \vre^{\dim U^-+\dim Z'+\dim U^+}\int_{A^{\delta+c_3\vre}_{t+c_4}} \det(\Ad(a)|_{U^+})\,da.
$$
The last integral was estimated in \cite[Proposition~3.22]{GOS2} (see also \cite[Proposition~3.8]{GOS2}).
We have
\begin{align*}
\int_{A^{\delta+c_3\vre}_{t+c_4}} \det(\Ad(a)|_{U^+})\,da \ll (\delta+c_3\vre)
\vol(S_t(\Phi,\Psi)).
%\ll \delta \vol(S_t(\Phi,\Psi)).
\end{align*}
Hence, 
$$
\vol(\mathcal{O}_{c_1\vre}
A^{\delta}_{t+c} \mathcal{O}_{c_1\vre})\ll (\delta+\vre) \vre^{\dim G-\dim A} \vol(S_t(\Phi,\Psi)).
$$
Since
$$
IJ\ll \vre^{-(\dim K+\dim H-\dim M)}=\vre^{-(\dim G-\dim A)},
$$
we deduce from \eqref{eq:sing} that
\begin{equation*}
\vol(\mathcal{O}_\vre S^\delta_t(\Phi,\Psi)\mathcal{O}_\vre)\ll (\delta+\vre) \vol(S_t(\Phi,\Psi))
\ll \vre^{1/3} \vol(S_t(\Phi,\Psi)),
\end{equation*}
as claimed.

Let $\tilde S^\delta_t(\Phi,\Psi)$ denote the subset of $\delta$-regular elements in
$S_t(\Phi,\Psi)$. Since $\Phi$ and $\Psi$ are bounded,
it follows from Proposition \ref{th:wave} that there exists $c>0$ such that for every $kah\in \tilde S^\delta_t(\Phi,\Psi)$,
\begin{align*}
\mathcal{O}_\vre kah \mathcal{O}_\vre \subset 
k \mathcal{O}_{c\vre} a \mathcal{O}_{c\vre}h
 &\subset  k (K\cap \mathcal{O}_{\ell c\vre}) (A\cap \mathcal{O}_{\ell c\vre}) a (H\cap
 \mathcal{O}_{\ell c\vre})h \\
&\subset  (K\cap \mathcal{O}_{\ell c^2\vre})k (A\cap \mathcal{O}_{\ell  c\vre}) a h (H\cap
 \mathcal{O}_{\ell c^2\vre}).
\end{align*}
Using that for some $c_1>0$ we have $\mathcal{O}_\vre S_t\subset S_{t+c_1\vre}$, we deduce 
from the previous estimate that
$$
\mathcal{O}_\vre \tilde S^\delta_t(\Phi,\Psi)\mathcal{O}_\vre\subset 
\tilde S^{\delta-\ell c\vre}_{t+2c_1\ell c^2\vre}( (K\cap \mathcal{O}_{\ell c^2\vre})\Phi , \Psi (H\cap
\mathcal{O}_{\ell c^2\vre})).$$
Then by the uniqueness properties of the Cartan decomposition,
\begin{align}\label{eq:reg}
\mathcal{O}_\vre \tilde S^\delta_t(\Phi,\Psi)\mathcal{O}_\vre -
S_t(\Phi,\Psi)\subset & S_{t+2c_1\ell c^2\vre}( (K\cap
\mathcal{O}_{\ell c^2\vre})\Phi-\Phi ,\Psi (H\cap \mathcal{O}_{\ell c^2\vre}))\\
&\cup
(S_{t+2c_1\ell c^2\vre}(\Phi,\Psi)
- S_{t}(\Phi,\Psi)) \nonumber \\
&\cup
S_{t+2c_1\ell c^2\vre}(\Phi,\Psi (H\cap \mathcal{O}_{\ell c^2\vre}) - \Psi).\nonumber
\end{align}
With respect to the Cartan decomposition $G=KA^+H$, the Haar measure on $G$ is given by
$dk\,\xi(a)da\, dh$ where $dk,da,dh$ are Haar measures on the components, and $\xi$ is
an explicit continuous function. There exists $c_2>0$ such that
\begin{align*}
&S_{t+2c_1\ell c^2\vre}\left((K\cap
\mathcal{O}_{\ell c^2\vre})\Phi-\Phi ,\Psi (H\cap \mathcal{O}_{\ell c^2\vre})\right)\\
\subset & 
((K\cap \mathcal{O}_{\ell c^2\vre})\Phi -\Phi)\,
(S_{t+c_2}\cap A^+)\,\Psi (H\cap \mathcal{O}_{\ell c^2\vre}).
\end{align*}
Hence, it follows that
\begin{align*}
&\vol(S_{t+2c_1\ell c^2\vre}\left((K\cap \mathcal{O}_{\ell c^2\vre})\Phi -\Phi, \Psi(H\cap
\mathcal{O}_{\ell c^2\vre}))\right)\\
\ll & \vol\left((K\cap \mathcal{O}_{\ell c^2\vre})\Phi -\Phi\right)\int_{S_{t+c_2}\cap A^+} \xi(a)\,da\\
\ll & (\ell \vre) \vol(S_t(\Phi,\Psi))
\ll  \vre^{1/3} \vol(S_{t+c_3}(\Phi,\Psi))
\end{align*}
for some $c_3>0$, where we used \eqref{eq:dl}.
Similarly,
\begin{align*}
\vol(S_{t+2c_1\ell c^2\vre}(\Phi,\Psi(H\cap \mathcal{O}_{\ell c^2\vre})- \Psi))
\ll  \vre^{1/3} \vol(S_{t+c_3}(\Phi,\Psi)).
\end{align*}
We will use the Lipschitz property of the function
$\phi(t):=\int_{\log \|ka\|\le t} \xi(a)da$: for sufficiently large $t$ and  $\vre\in (0,1)$,
$$
\phi(t+\vre)-\phi(t)\ll \vre\, \phi(t)
$$
uniformly on $k$. This property can be proved using the argument from
\cite[Appendix]{EMS} --- see \cite[Proposition 7.3]{GN}.
We obtain
\begin{align*}
&\vol(S_{t+2c_1\ell c^2\vre}(\Phi,\Psi)
- S_{t}(\Phi,\Psi))\\
\ll & \int_{K/M} \left(\int_{t\le \log \|kav_0\|\le t+2c_1\ell c^2\vre} \xi(a)da\right) dk\\
\ll & \ell\vre \int_{K/M} \left(\int_{\log \|ka v_0\|\le t} \xi(a)da\right)dk\ll \vre^{1/3} \vol(S_t(\Phi,\Psi)).
\end{align*}
Now it follows from (\ref{eq:reg}) that
$$
\vol(\mathcal{O}_\vre \tilde S^\delta_t(\Phi,\Psi)\mathcal{O}_\vre -
S_t(\Phi,\Psi))\ll \vre^{1/3}\vol(S_t(\Phi,\Psi)).
$$
Combining this estimate with \eqref{eq:11}, we deduce that 
$$
\vol(\mathcal{O}_\vre S_t(\Phi,\Psi)\mathcal{O}_\vre -
S_t(\Phi,\Psi))\ll \vre^{1/3}\vol(S_t(\Phi,\Psi)).
$$
Similarly, one shows that
$$
\vol\left( S_t(\Phi,\Psi) -
\cap_{u,v\in \mathcal{O}_\vre} S_t(\Phi,\Psi)\right)\ll \vre^{1/3}\vol(S_t(\Phi,\Psi)).
$$
Hence, the sets $S_t(\Phi,\Psi)$ are H\"older well-rounded with exponent $1/3$.
\end{proof}
 
 \subsection{Completion of the proofs}
We now turn to complete the proofs of the results stated in \S 8.1. 
\begin{proof}[Proof of Theorem \ref{c:s22}]
By Proposition \ref{th:sec_sym}, the sets $S_t(\Phi,\Psi)$ are H\"older well-rounded
with exponent $1/3$, and by Theorem \ref{semisimple mean}, the uniform averages 
supported on  $S_t(\Phi,\Psi)$ satisfy the stable quantitative
mean ergodic theorem. Hence, the Theorem \ref{c:s22} is a consequence of Theorem \ref{error estimate}.
\end{proof}

\begin{proof}[Proof of Corollary \ref{c:s2}]
%We fix a right-invariant Riemannian metric on $G$.
%The neighbourhoods $\mathcal{O}_\vre$ are defined with respect to this metric.
%Passing to a finite index subgroup of $\Gamma$, we may assume that $\Gamma\cap M=\{e\}$.
Let $d$ be a right-invariant Riemannian metric on $M\backslash H$ 
and $x_0\in M\backslash H$ with trivial $(\Gamma\cap H)$-stabiliser. Define
$$
\mathcal{D}_r=\{x\in M\backslash H:\, d(x,x_0)\le d(x,x_0\gamma)+r\quad\hbox{for $\gamma\in \Gamma\cap H$}\}.
$$
The set $\mathcal{D}_0$ is the Dirichlet domain for the right $(\Gamma\cap H)$-action on $H$.
Since $H/(\Gamma\cap H)$ is compact, $\mathcal{D}_0$ is compact,
and it can be defined by finitely many inequalities:
$$
\mathcal{D}_0=\{x\in M\backslash H:\, d(x,x_0)\le d(x,x_0\gamma_i)\quad\hbox{for $i=1,\ldots, k$}\}.
$$

Note that $\mathcal{D}_0$ satisfies
$$
\mathcal{D}_0\Gamma=M\backslash H\quad\hbox{and}\quad \hbox{int}(\mathcal{D}_0)\gamma_1\cap\hbox{int}(\mathcal{\mathcal{D}}_0)\gamma_2=\emptyset
\quad\hbox{for $\gamma_1\ne \gamma_2$}.
$$

We choose a measurable fundamental
domain $\mathcal{D}$ for the $(\Gamma\cap H)$-action on $M\backslash H$ such that $\hbox{int}(\mathcal{D}_0)\subset \mathcal{D}\subset \mathcal{D}_0$.
Note that the map $KA^+/M\to G/H$ is one-to-one on the set of regular elements.
Let $\Sigma$ be a measurable section of the map $KA^+/M\to G/H$ which contains all regular elements.

We set  
$$
T_t(\Phi,\mathcal{D})=S_t(\Phi,H)\cap \Sigma (\mathcal{D}).
$$

We claim that
\begin{equation}\label{eq:eq}
|\Gamma H\cap B_t(\Phi)|=|\Gamma \cap T_t(\Phi,\mathcal{D})|.
\end{equation}
Clearly, for $g\in T_t(\Phi,\mathcal{D})$, we have $gH\in B_t(\Phi)$, and 
every $x\in \Gamma H\cap B_t(\Phi)$ is of the form $\gamma H$ for some $\gamma\in \Gamma \cap T_t(\Phi,H)$. Moreover, since $H=\mathcal{D}(H\cap\Gamma)$, we can choose $\gamma\in
\Gamma \cap T_t(\Phi,\mathcal{D})$. Hence, it remains to show that
if $\gamma_1H=\gamma_2H$ for some $\gamma_1,\gamma_2\in \Gamma \cap T_t(\Phi,\mathcal{D})$,
then  $\gamma_1=\gamma_2$. We have $\gamma_i=\omega_ih_i\in \Sigma \mathcal{D}$.
It follows from the definition of $\Sigma$ that $\omega_1=\omega_2$.
Hence, $h_1=h_2(\gamma_2^{-1}\gamma_1)$, and because $h_1,h_2$ are both in the fundamental
domain $\mathcal{D}$, we conclude that $h_1=h_2$.
 
Next, we show that the sets  $T_t(\Phi,\mathcal{D})$ are H\"older well-rounded with exponent $1/3$.
By Proposition \ref{th:sec_sym} (see also Remark \ref{r:sec}), it remains to check that the set $\mathcal{D}$ 
is Lipschitz well-rounded, namely,
satisfies
$$
\vol\left( \mathcal{D}(\mathcal{O}_\vre\cap H) -\cap_{u\in \mathcal{O}_\vre\cap H} \mathcal{D} u\right)\ll \vre. 
$$
For $\gamma\ne e$, the function $f_\gamma(x)=d(x,x_0)-d(x,x_0\gamma)$ is regular on $\{f_\gamma=0\}$.
Hence, for a compact set $\Omega\subset H$,
\begin{equation}\label{eq:D}
\vol(\{x\in\Omega:\, -\vre< f_\gamma(x)<\vre\})\ll_{\Omega} \vre.
\end{equation}
Also, for $h\in \mathcal{O}_\vre\cap H$, 
$$
|f_\gamma(xh)-f_\gamma(x)|\ll \vre
$$
uniformly on $x$ in compact sets.
This implies that for some $c>0$,
\begin{align*}
\mathcal{D}(\mathcal{O}_\vre\cap H) -\cap_{u\in \mathcal{O}_\vre\cap H} \mathcal{D} u 
&\subset  \mathcal{D}_0\overline{(\mathcal{O}_{\vre_1}\cap H)} \cap ( \mathcal{D}_{c\vre}-\mathcal{D}_{-c\vre})\\
&\subset \bigcup_{i=1}^m 
\{h\in  \mathcal{D}_0\overline{(\mathcal{O}_{\vre_1}\cap H)} :\, -c\vre< f_{\gamma_i}(h)<c\vre\},
\end{align*}
Hence, it follows from (\ref{eq:D}) that the set $S$ is Lipschitz well-rounded.
Then by Proposition \ref{th:sec_sym}, the sets $T_t(\Phi,\mathcal{D})$
are H\"older well-rounded with exponent $1/3$. Now the corollary follows from (\ref{eq:eq}) 
and Theorem \ref{c:s22}. 
\end{proof}

\end{document}